\documentclass[a4paper]{amsart}

\usepackage[english]{babel}

\usepackage{amssymb}
\usepackage{mathrsfs}
\usepackage{hyperref}
\usepackage{xcolor}
\usepackage{eucal}

\usepackage{amssymb}
\usepackage{amsmath}
\usepackage{amsthm}
\usepackage{imakeidx}
\usepackage{breqn}

\usepackage{diagbox}

\usepackage{verbatim}

\usepackage{fancyhdr}

\usepackage{tikz-cd}

\usepackage[utf8]{inputenc}

\usepackage{graphicx}

\usepackage{xcolor}

\usepackage[style=numeric, url=false, doi=false, backend=bibtex]{biblatex}

\theoremstyle{definition}

\newtheorem{mydef}{Definition}[section]

\theoremstyle{remark}

\newtheorem{mybem}[mydef]{Remark}

\newtheorem{myobs}[mydef]{Observation}

\theoremstyle{plain}

\newtheorem{mycol}[mydef]{Corollary}
\newtheorem{mysen}[mydef]{Theorem}
\newtheorem{mylem}[mydef]{Lemma}

\newtheorem*{myclaim}{Claim}

\newtheorem*{mysclaim}{Subclaim}

\newtheorem{mysenx}{Theorem}

\numberwithin{mydef}{section}

\DeclareMathOperator{\cof}{cof}
\DeclareMathOperator{\dom}{dom}
\DeclareMathOperator{\GCH}{GCH}
\DeclareMathOperator{\ZFC}{ZFC}
\DeclareMathOperator{\supp}{supp}

\DeclareMathOperator{\im}{im}

\DeclareMathOperator{\SR}{SR}
\DeclareMathOperator{\otp}{otp}
\DeclareMathOperator{\cf}{cf}

\DeclareMathOperator{\Add}{Add}

\DeclareMathOperator{\Coll}{Coll}

\DeclareMathOperator{\MM}{MM}
\DeclareMathOperator{\Suc}{Suc}

\DeclareMathOperator{\RCS}{RCS}

\DeclareMathOperator{\tcl}{tcl}

\DeclareMathOperator{\SRP}{SRP}

\DeclareMathOperator{\CH}{CH}

\DeclareMathOperator{\Hull}{Hull}
\DeclareMathOperator{\RwD}{RwD}
\DeclareMathOperator{\CU}{CU}

\newcommand{\dI}{\mathbb{I}}

\newcommand{\dP}{\mathbb{P}}
\newcommand{\dQ}{\mathbb{Q}}
\newcommand{\dR}{\mathbb{R}}

\newcommand{\bI}{\mathbf{I}}

\newcommand{\uhr}{\upharpoonright}

\title{On Friedman's Property}
\author{Hannes Jakob}

\subjclass[2020]{03E05, 03E35, 03E55} 

\date{\today}

\begin{document}
	
	
	\baselineskip=17pt
	\keywords{Friedman Property, Strategic Closure, Elementary Embeddings} 
	
	
	\begin{abstract}
		We define forcing orders which add witnesses to the failure of various forms of Friedman's Property. These posets behave similarly to the forcing order adding a nonreflecting stationary set but have the advantage of allowing the construction of master conditions and thus the preservation of various large cardinal properties. We apply these new techniques to separate various instances of variants of Friedman's Problem, both between different instances at one cardinal as well as equal instances at different cardinals and en passant obtain some new results regarding the differences between ${<}\,\kappa$- and $\kappa$-strategic closure.
	\end{abstract}
	
	\maketitle
	
	In \cite{FriedmanClosed}, Harvey Friedman proved the important Theorem that any stationary subset of $\omega_1$ contains closed subsets of arbitrarily long, countable, ordertype. In that same paper he introduces the property $F(\kappa)$ for $\kappa\geq\omega_2$, stating that any subset of $\kappa$ either contains or is disjoint from a closed set of ordertype $\omega_1$. We note that $\omega_1$ clearly is the critical ordinal since the same statement for ordinals ${<}\,\omega_1$ is provable in $\ZFC$ and the same statement for ordinals ${>}\,\omega_1$ is inconsistent (take $A:=E_{\omega}^{\kappa}$). Later on, the stronger property $F^+(\kappa)$, stating that every stationary subset of $E_{\omega}^{\kappa}$ contains a closed subset of ordertype $\omega_1$, was shown by Shelah in \cite{ShelahProperImproper} (Chapter XI, Theorem 7.1) to be consistent from a Mahlo cardinal and in \cite{ForemanMagidorShelahMM} and \cite{FengJechProjectiveStationary} to follow from $\MM$ and $\SRP$ respectively. It was also noticed by Silver (as stated in \cite{FriedmanClosed}) that forcing the failure of $F(\kappa)$ is quite easy, as any model of the form $V[\Coll(\omega,\omega_1)]$ will satisfy $\neg F(\kappa)$ for any $\kappa>\omega_1$.
	
	In this work, we will introduce posets which add witnesses to the failure of various variants of $F$ and $F^+$ in a more gentle manner in order to separate many instances of $F$ and $F^+$. These posets are modelled after the poset adding a nonreflecting stationary set but have the crucial advantage that they allow the building of master conditions and thus the lifting of ground-model embeddings. This is used in two ways: By forcing with our new posets over the standard model of Martin's Maximum, we obtain a model where some form of $F$ or $F^+$ (and thus also $\MM$) fails, but Martin's Maximum holds for all posets which do not imply the corresponding form of $F$ or $F^+$. Using this we will obtain a perfect separation between instances of $F$ or $F^+$ at $\omega_2$ (i.e. we prove implications between instances of $F$ and $F^+$ and show that all other implications are not provable in $\ZFC$). We will also use our posets to obtain a separation between instances of $F$ or $F^+$ at different cardinals. Here we can show that, despite their similarities, $F$ and $F^+$ have different ``Meta-properties'', e.g. the failure of $F$ is compact at weakly compact cardinals (i.e. if $F(\delta)$ fails for all $\delta<\kappa$ then $F(\kappa)$ fails if $\kappa$ is weakly compact) while $F^+$ can consistently hold for the first time at a weakly compact cardinal.
	
	The paper is organized as follows: In the first section, we define the notions that will appear in this paper and give a more detailed introduction into our results. In section 2, we give an easier proof of Shelah's Theorem that $F^+(\omega_2)$ is consistent from a Mahlo cardinal. In section 3, we introduce our posets to add witnesses to the failure of $F$ and $F^+$. In section 4, we introduce the new maximal versions of $\MM$ and prove their relative consistency. In sections 5 and 6, we use these axioms to obtain separations between instances of $F$ and $F^+$ respectively. In section 7, we show that we can force $F$ and $F^+$ to fail in a more gentle way while preserving large cardinals and obtain results regarding the compactness of the failure of $F$ and $F^+$ at large cardinals.
	
	\section{Introduction and Preliminaries}
	
	We assume the reader has some familiarity with the basic theory of club and stationary sets as well as a solid forcing knowledge. Our notation is standard, we require filters to be upwards closed (so that $q\leq p$ implies that $q$ forces more than $p$). For any ordinal $\gamma$ we let $E_{\omega}^{\gamma}$ denote the set of all ordinals $\delta\in\gamma$ with $\cf(\delta)=\omega$. A function is said to be \emph{normal} if it is strictly increasing and continuous.
	
	We first define our variants of $F$ and $F^+$:
	
	\begin{mydef}
		Let $\kappa\geq\omega_2$ be regular.
		\begin{enumerate}
			\item Let $\lambda\leq\kappa$ be a cardinal. The property $F(\lambda,\kappa)$ states that whenever $g\colon\kappa\to\lambda$ is regressive there is a normal function $f\colon\omega_1\to\kappa$ such that $g\circ f$ is constant.
			\item Let $(D_i)_{i\in\omega_1}$ be a partition of $\omega_1$ (we allow $D_i=\emptyset$ for some $i$). The property $F^+((D_i)_{i\in\omega_1},\kappa)$ states that whenever $(A_i)_{i\in\omega_1}$ is a sequence of stationary subsets of $E_{\omega}^{\kappa}$ there is a normal function $f\colon\omega_1\to\kappa$ such that $f[D_i]\subseteq A_i$ for all $i\in\omega_1$.
		\end{enumerate}
	\end{mydef}
	
	These properties are of course related to the principle of \emph{stationary reflection}: We let $\SR(E_{\omega}^{\kappa})$ state that for any stationary $S\subseteq E_{\omega}^{\kappa}$ there is some $\gamma\in\kappa$ with uncountable cofinality such that $S\cap\gamma$ is stationary in $\gamma$.
	
	For $F(\lambda,\kappa)$, it is clear that $F(\lambda,\kappa)$ implies $F(\lambda',\kappa)$ whenever $\lambda'<\lambda$ (as any function from $\kappa$ to $\lambda'$ is also a function from $\kappa\to\lambda$). We will show that this is sharp:
	
	\begin{mysenx}
		Assume $\kappa$ is supercompact and $\lambda<\kappa$ ($\lambda$ may be finite). There is a forcing extension in which $\kappa=\omega_2$, $F(\lambda,\kappa)$ holds and $F(\lambda',\kappa)$ fails for any $\lambda'>\lambda$.
	\end{mysenx}
	
	So it is e.g. relatively consistent that any partition of $\omega_2$ into $42$ pieces has at least one part that includes a closed copy of $\omega_1$ while there is a partition of $\omega_2$ into $43$ sets such that no part includes a closed copy of $\omega_1$.
	
	To state our results for $F^+$, we first need to introduce more notions.
	
	\begin{mydef}
		$\MM$ states that whenever $\dP$ preserves stationary subsets of $\omega_1$ and $(D_i)_{i\in\omega_1}$ is a sequence of open dense subsets of $\dP$ there is a filter $G\subseteq\dP$ such that $G\cap D_i\neq\emptyset$ for any $i\in\omega_1$.
	\end{mydef}
	
	A very important consequence of $\MM$ which captures many of its implications is the principle $\SRP$, introduced by Todorcevic in a handwritten note in 1987. Later on, an equivalent, but simpler, characterization was found by Feng and Jech in \cite{FengJechProjectiveStationary}. Recently, in \cite{FuchsCanonicalFragments}, Gunter Fuchs introduced the parametrized variants.
	
	\begin{mydef}
		Let $\Theta$ be a regular cardinal. $S\subseteq[H(\Theta)]^{<\omega_1}$ is \emph{projective stationary} if for any stationary $A\subseteq\omega_1$, the set
		$$\{M\in S\;|\;M\cap\omega_1\in A\}$$
		is stationary in $[H(\Theta)]^{<\omega_1}$.
		
		The \emph{strong reflection principle} $\SRP$ states that whenever $S$ is projective stationary there is an increasing and continuous $\in$-chain $(N_i)_{i\in\omega_1}$ of members of $S$.
		
		If $\mathcal{S}$ is a collection of projective stationary subsets of $[H(\Theta)]^{<\omega_1}$, $\SRP(\mathcal{S})$ states that whenever $S\in\mathcal{S}$ there is an increasing and continuous $\in$-chain $(N_i)_{i\in\omega_1}$ of members of $S$.
	\end{mydef}
	
	Given a partition $(D_i)_{i\in\omega_1}$ of $\omega_1$ and a sequence $(A_i)_{i\in\omega_1}$ of stationary subsets of $E_{\omega}^{\kappa}$ there is (for $\Theta\geq\kappa$) a canonical subset of $[H(\Theta)]^{<\omega_1}$ associated with the pair $((D_i)_{i\in\omega_1},(A_i)_{i\in\omega_1})$: We let
	\begin{multline*}
		S((D_i)_{i\in\omega_1},(A_i)_{i\in\omega_1},\Theta):=\\
		\{M\in[H(\Theta)]^{<\omega_1}\;|\;\forall i\in\omega_1(M\cap\omega_1\in D_i\to\sup(M\cap\kappa)\in A_i)\}
	\end{multline*}
	
	By an argument of Fuchs from \cite{FuchsCanonicalFragments}, if there is an increasing and continuous $\in$-chain $(N_i)_{i\in\omega_1}$ of members of $S((D_i)_{i\in\omega_1},(A_i)_{i\in\omega_1},\Theta)$ for some $\Theta$ then there is a normal function $f\colon\omega_1\to\kappa$ such that $f[D_i]\subseteq A_i$ (by simply letting $f(N_i\cap\omega_1):=\sup(N_i\cap\kappa)$ and ``filling in the gaps''). We will later show, answering a question of Fuchs from the same paper, that $S((D_i)_{i\in\omega_1},(A_i)_{i\in\omega_1},\Theta)$ is always projective stationary, which means that $\SRP$ implies $F^+((D_i)_{i\in\omega_1},\kappa)$ for any partition $(D_i)_{i\in\omega_1}$ of $\omega_1$ and any regular cardinal $\kappa\geq\omega_2$. We will also define $\mathcal{S}((D_i)_{i\in\omega_1},\kappa,\Theta)$ as the collection of $S((D_i)_{i\in\omega_1},(A_i)_{i\in\omega_1},\Theta)\cap C$ for any sequence $(A_i)_{i\in\omega_1}$ of stationary subsets of $E_{\omega}^{\kappa}$ and any club $C\subseteq [H(\Theta)]^{<\omega_1}$. In section 6 we will show, answering another question of Fuchs from \cite{FuchsCanonicalFragments} in a strong way, that the ordering $\leq^*$ on the set of partitions (defined by $\overline{D}\leq^*\overline{E}$ if $\overline{D}$ refines $\overline{E}$ on a club) characterizes exactly when there is an implication between $\SRP(\mathcal{S}((D_i)_{i\in\omega_1}),\kappa,\Theta)$ and $F^+((E_i)_{i\in\omega_1})$:
	
	\begin{mysenx}
		Let $\overline{D}$ be a partition of $\omega_1$.
		\begin{enumerate}
			\item If $\overline{E}$ is a partition of $\omega_1$ with $\overline{E}\leq^*\overline{D}$ and $\Theta\geq\kappa$ is regular, $\SRP(\mathcal{S}(\overline{E},\kappa,\Theta))$ implies $F^+(\overline{D})$.
			\item If we are in the standard model for $\MM$, there is a forcing extension in which $F^+(\overline{D})$ fails and $\SRP(\mathcal{S}(\overline{E},\kappa,\Theta))$ holds for all partitions $\overline{E}$ with $\overline{E}\not\leq^*\overline{D}$ and all regular $\Theta\geq\kappa$.
		\end{enumerate}
	\end{mysenx}
	
	In the last section we will concern ourselves with the effect large cardinal properties can have on the pattern of $F$ and $F^+$. We will show that there is a natural separation between $F$ and $F^+$ caused by the different complexity ($\neg F(\kappa)$ is a $\Pi_1^1$-statement while $\neg F^+$ is more complex due to the requirement of stationarity). This is also reflected in the different closure properties of the corresponding forcings (a similar situation occurs in the separation between the existence of a square sequence and a nonreflecting stationary set, see Remark 6.7 in Cummings' chapter in the Handbook of Set Theory):
	
	\begin{mydef}
		Let $\dP$ be a poset and $\gamma$ an ordinal. The \emph{completeness game} $G(\dP,\gamma)$ on $\dP$ of length $\gamma$ is played by COM and INC as follows: COM plays at even ordinals (including limits) and INC plays at odd ordinals. The first move by COM is $1_{\dP}$. Afterwards, given a sequence $(p_{\alpha})_{\alpha<\delta}$ for $\delta<\gamma$, the player whose turn it is needs to play a condition $p_{\delta}$ such that $p_{\delta}\leq p_{\alpha}$ for any $\alpha<\delta$. COM wins if they have a legal move at every ordinal below (but not including) $\gamma$, otherwise INC wins.
		
		We say that a forcing order is $\gamma$-strategically closed if COM has a winning strategy in $G(\dP,\gamma)$. It is ${<}\,\gamma$-strategically closed if it is $\delta$-strategically closed for all $\delta<\gamma$.
	\end{mydef}
	
	Another related property is that of strong distributivity, see \cite[Definition 3.1]{JakobDisjointInterval}:
	
	\begin{mydef}
		Let $\dP$ be a forcing order and $\kappa$ a cardinal. $\dP$ is \emph{strongly ${<}\,\kappa$-distributive} if whenever $(D_{\alpha})_{\alpha<\kappa}$ is a sequence of open dense subsets of $\dP$ and $p\in\dP$ there is a descending sequence $(p_{\alpha})_{\alpha<\kappa}$ such that $p_0\leq p$ and $p_{\alpha}\in D_{\alpha}$ for any $\alpha<\kappa$.
	\end{mydef}
	
	By \cite[Theorem 3.6]{JakobDisjointInterval}, this is equivalent to the statement that INC does not have a winning strategy in $G(\dP,\kappa)$.
	
	The forcing adding a witness to $\neg F^+(\kappa)$ will always be $\kappa$-strategically closed while the forcing adding a witness to $\neg F(\kappa)$ cannot even by strongly ${<}\,\kappa$-distributive assuming $F(\kappa)$ holds in the ground model (however it will be ${<}\,\kappa$-strategically closed assuming $F$ fails below $\kappa$). While the two properties might seem very close in strength we will show that there are many natural regularity properties implied by $\kappa$-strategic closure (or even strong ${<}\,\kappa$-distributivity) which do not hold for the forcing adding a witness to $\neg F(\kappa)$ (such as not adding branches to $\kappa$-trees or allowing the building of sufficiently generic filters in the ground model).
	
	On the positive side, we have the following $\ZFC$ results:
	
	\begin{mysenx}
		Let $\kappa$ be a cardinal.
		\begin{enumerate}
			\item Assume $F(\delta)$ fails for all $\delta<\kappa$.
			\begin{enumerate}
				\item If $\kappa$ is weakly compact, $F(\kappa)$ fails.
				\item If $\lambda\geq\kappa$ is such that $\kappa$ is $\lambda$-supercompact, $F(\lambda^+)$ fails.
			\end{enumerate}
			\item Assume $F^+(\delta)$ fails for all $\delta<\kappa$ and $\lambda\geq\kappa$ is such that $\kappa$ is $\lambda$-supercompact. Then $F^+(\delta)$ fails for all $\delta\leq\lambda$.
		\end{enumerate}
	\end{mysenx}
	
	So certain large cardinal properties have a different effect on the patterns of $F$ and $F^+$ respectively. This is contrasted by the following independence results:
	
	\begin{mysenx}
		Assuming enough large cardinals, the following is consistent:
		\begin{enumerate}
			\item There exists a Mahlo cardinal $\kappa$ such that $F(\kappa)$ holds and $F(\delta)$ fails for all $\delta<\kappa$.
			\item There exists a weakly compact cardinal $\kappa$ such that $F(\kappa^+)$ holds and $F(\kappa)$ fails.
			\item There exist cardinals $\kappa\leq\lambda$ such that $\kappa$ is $\lambda$-supercompact, $F(\lambda^{++})$ holds and $F(\lambda^+)$ fails.
			\item There exists a weakly compact cardinal $\kappa$ such that $F^+(\kappa)$ holds and $F^+(\delta)$ fails for all $\delta<\kappa$.
			\item There exist cardinals $\kappa\leq\lambda$ such that $\kappa$ is $\lambda$-supercompact, $F^+(\lambda^+)$ holds and $F^+(\delta)$ fails for all $\delta\leq\lambda$.
		\end{enumerate}
	\end{mysenx}
	
	\section{On the consistency of $F^+(\omega_2)$}
	
	In this section we give an easier proof of a result due to Shelah that $F^+((D_i)_{i\in\omega_1},\omega_2)$ is consistent from a Mahlo cardinal, where $(D_i)_{i\in\omega_1}$ is the trivial partition given by $D_0:=\omega_1$. For simplicity, we denote this by $F^+(\omega_2)$.
	
	\subsection*{Shelah's $S$-condition}\hfill
	
	There is a very natural forcing which, given some stationary $A\subseteq E_{\omega}^{\kappa}$, adds a closed set $c\subseteq A$ with ordertype $\omega_1$
	
	\begin{mydef}
		Let $\kappa$ be a regular cardinal and $A\subseteq E_{\omega}^{\kappa}$ stationary. The forcing $\dP(A)$ is defined as follows: $p\in\dP(A)$ if $p\colon\alpha+1\to A$ is a normal function for some $\alpha<\omega_1$. $q\leq p$ in $\dP(A)$ if $q$ end-extends $p$.
	\end{mydef}
	
	In most cases, this poset collapses $\kappa$ to $\omega_1$ and is not proper (assuming $E_{\omega}^{\kappa}\smallsetminus A$ is stationary, as this is destroyed). We do not know if this poset is always semiproper, but it is easy to see that it preserves stationary subsets of $\omega_1$ (this was first shown in \cite{ForemanMagidorShelahMM}):
	
	\begin{mylem}\label{PAStatPres}
		Let $\kappa\geq\omega_2$ be regular and $A\subseteq E_{\omega}^{\kappa}$ be stationary. Then $\dP(A)$ preserves stationary subsets of $\omega_1$.
	\end{mylem}
	
	\begin{proof}
		Let $S\subseteq\omega_1$ be stationary and $\dot{C}$ a $\dP(A)$-name for a club. Let $\Theta$ be large enough and $F\colon[H(\Theta)]^{<\omega}\to[H(\Theta)]^{<\omega_1}$ a function such that any $M\in[H(\Theta)]^{<\omega_1}$ that is closed under $F$ is an elementary submodel of $(H(\Theta),\in)$ containing $\dot{C}$ and $\dP(A)$. Let $M\in[H(\Theta)]^{\omega_1}$ be closed under $F$ with $M\cap\omega_2\in A$. Let $\{\delta_n\;|\;n\in\omega\}\subseteq M\cap\omega_2$ be unbounded in $\sup(M\cap\omega_2)$. We can find $N\subseteq M$ such that $N\in[H(\Theta)]^{<\omega_1}$ is closed under $F$, contains $\{\delta_n\;|\;n\in\omega\}$ as a subset and $N\cap\omega_1\in S$. Then $\sup(N\cap\omega_2)=\sup(M\cap\omega_2)\in A$, $N\prec (H(\Theta),\in)$ and $\dot{C},\dP(A)\in N$. Let $(D_n)_{n\in\omega}$ enumerate all open dense subsets of $\dP(A)$ lying in $N$. In $N$, construct a descending sequence $(p_n)_{n\in\omega}$ such that $p_n\in D_n$. It follows that $p:=\bigcup_np_n\cup\{(N\cap\omega_1,\sup(N\cap\omega_2))\}\in\dP(A)$ and forces that $\dot{C}\cap(N\cap\omega_1)$ is unbounded in $N\cap\omega_1$, so $N\cap\omega_1\in S\cap\dot{C}$.
	\end{proof}
	
	In \cite{ShelahProperImproper}, chapter XI, Shelah introduces a condition known as the $S$- (or $\bI$-) condition. This condition holds for $\dP(A)$ and is iterable with revised countable support (we will not go into that definition here). Moreover, forcing notions satisfying the condition for suitable parameters do not add reals and preserve $\omega_1$.
	
	\begin{mydef}
		Let $\bI$ be a family of ideals. An \emph{$\bI$-tagged tree} is a pair $(T,\dI)$ such that:
		\begin{enumerate}
			\item $T$ is a tree, i.e. a nonempty set of finite sequences of ordinals closed under restriction.
			\item $\dI$ is a function from $T$ into $\bI$ such that for any $t\in T$,
			$$\Suc_T(t):=\{s\;|\; s\text{ is an immediate successor of }t\text{ in }T\}\subseteq \dom(\dI(t))$$
			\item For every $t\in T$ we have $\Suc_T(t)\neq\emptyset$ and for every $t\in T$ there is $t'\in T$ with $t\trianglelefteq t'$ and $\Suc_T(t')\notin\dI(t)$.
		\end{enumerate}
	\end{mydef}
	
	The last condition states that our tree is a \emph{Miller-Style tree}, i.e. we allow branches to not split for long periods of time. A different possibility (not explored here and more complicated) uses \emph{Laver-Style trees}, i.e. trees which have a stem and split everywhere above that stem.
	
	We define the following notions: Given two $\bI$-tagged trees $(T,\dI)$ and $(T',\dI')$ we write $T'\leq^*T$ if $T'\subseteq T$, $\dI\uhr T'=\dI'$ and for any $t\in T'$, if $\Suc_T(t')\notin\dI(t)$, then $\Suc_{T'}(t')\notin\dI'(t)$ (i.e. $T'$ splits at the same nodes as $T$). Given any tree $T$, we write $[T]$ for the set of all functions $b$ on $\omega$ such that $b\uhr k\in T$ for any $k\in\omega$. It is clear that $[T]$ is nonempty for any $\bI$-tagged tree (as every node has at least one successor).
	
	\begin{mydef}
		Let $\bI$ be a family of ideals and $\dP$ a forcing notion. $\dP$ \emph{satisfies the $\bI$-condition} if there is a function $F$ (with correct domain and image) such that for every $\bI$-tagged tree $(T,\dI)$: If $f=(f^I,f^{II})\colon T\to\dP^2$ is a function such that
		\begin{enumerate}
			\item $t\trianglelefteq t'$ implies $f^I(t)\geq f^{II}(t)\geq f^{I}(t')$
			\item For any $t\in T$ we have
			$$(\Suc_T(t),(f^I(t'))_{t'\in\Suc_T(t)})=F(t,(f(t'))_{t'\trianglelefteq t})$$
		\end{enumerate}
		then whenever $T'\leq^*T$ there is $p\in\dP$ such that
		$$p\Vdash\exists b\in[T']\forall k\in\omega(f^I(b\uhr k)\in\dot{G})$$
	\end{mydef}
	
	This is a slight simplification of Shelah's original definition but it works just the same. Given a cardinal $\lambda$, let $\dI_{\lambda}$ be the bounded ideal on $\lambda$. A forcing notion $\dP$ \emph{satisfies the $S$-condition} for $S$ a set of regular cardinals if it satisfies the $\{\dI_{\lambda}\;|\;\lambda\in S\}$-condition.
	
	The important points are summarized as follows:
	
	\begin{mysen}
		Let $\dP$ be a forcing notion and $\bI$ a family of ideals such that any $\dI\in\bI$ is ${<}\,\omega_2$-complete. If $\dP$ satisfies the $\bI$-condition, then
		\begin{enumerate}
			\item (CH) forcing with $\dP$ does not add reals.
			\item forcing with $\dP$ does not collapse $\omega_1$.
		\end{enumerate}
	\end{mysen}
	
	\begin{mysen}
		Let $\kappa\geq\omega_2$ be regular and $A\subseteq E_{\omega}^{\kappa}$ be stationary. Then $\dP(A)$ satisfies the $\{\kappa\}$-condition.
	\end{mysen}
	
	For the iterability of the $\bI$-condition, Shelah provides the following template (we state it in its improved form from \cite{GitkShelahICondition}). We will not go into the definition of an $\RCS$ definition here. Roughly speaking, an $\RCS$ iteration is a countable support iteration where the support of a given condition does not contain ordinals but names for ordinals. This allows the iteration to take full support limits at ordinals which do not have cofinality $\omega$ in the ground model but only obtain this cofinality somewhere along the iteration (which is necessary if we are incorporating improper forcing notions such as Namba or Prikry forcing). A simpler definition of an $\RCS$ iteration which also suffices for the following Theorem can be found in unpublished work of Ulrich Fuchs (see \cite{FuchsDonderRCS}).
	
	\begin{mysen}
		Suppose
		\begin{enumerate}
			\item $\overline{\dP}:=(\dP_i,\dot{\dQ}_i)_{i<\alpha}$ is an RCS iteration,
			\item Whenever $i<\alpha$, there is some family $\bI_i$ (in $V$) such that
			\begin{enumerate}
				\item $\dP_i$ forces that every member of $\bI_i$ is ${<}\,\dot{\omega}_2$-complete
				\item $\dP_i$ forces that $\dot{\dQ}_i$ satisfies the $\check{\bI}_i$-condition
			\end{enumerate}
			\item If $i$ is singular or $|\dP_j|\geq i$ for some $j<i$ then there are $\lambda$ and $\mu$ such that every member of $\bigcup_{j\geq i}\bI_j$ is ${<}\,\lambda^+$-complete, whenever $I\in\bigcup_{j<i}\bI_j$, $|\bigcup I|<\mu$ and $\lambda=\lambda^{<\mu}$
			\item If $i$ is regular and $|\dP_j|<i$ for all $j<i$ then every member of $\bI_i$ is ${<}\,i$-complete
		\end{enumerate}
		Then the revised limit of $\overline{\dP}$ satisfies the $\bigcup_{i<\alpha}\bI_i$-condition.
	\end{mysen}
	
	Our aim now is to prove the following:
	
	\begin{mysen}\label{ShelahTheorem}
		Suppose $\kappa$ is a Mahlo cardinal. There is a forcing extension where $\kappa=\omega_2$ and $F^+(\kappa)$ holds.
	\end{mysen}
	
	This trivial partition is almost the maximum we can hope for: Assuming $(D_i)_{i\in\omega_1}$ is a partition such that at least two $D_i$ are stationary, $F^+((D_i)_{i\in\omega_1},\omega_2)$ implies that any two stationary subsets of $\omega_2$ reflect simultaneously which implies that $\omega_2$ is weakly compact in $L$ (see \cite{MagidorReflectionStat}).
	
	Shelah proves Theorem \ref{ShelahTheorem} by using a diamond sequence which guesses potential iterations of forcings of the form $\dP(A)$, followed by forcings to destroy those stationary sets which do not have closed subsets of ordertype $\omega_1$. However, this requires building master conditions without much knowledge about the exact iteration and is therefore rather technical. We will provide an easier proof by forcing the combination of stationary reflection and $\Diamond$ to hold at a Mahlo cardinal and then iterating forcings of the form $\dP(\dot{A})$ given by the aforementioned forced diamond sequence.
	
	\subsection*{Reflection with Diamond}\hfill

	\begin{mydef}
		Let $M$ be a set and $\kappa$ a cardinal. $M$ is a ${<}\,\kappa$-model if the following holds:
		\begin{enumerate}
			\item $M\prec H(\kappa^+)$.
			\item $\nu:=M\cap\kappa=|M|\in\kappa$ is inaccessible.
			\item ${}^{<\nu}M\subseteq M$.
		\end{enumerate}
	\end{mydef}
	
	For a $\kappa$-model $M$, we let $\pi_M\colon M\to N_M$ denote its Mostowski-Collapse.
	
	By a folklore result, being Mahlo is equivalent to the existence of many ${<}\,\kappa$-models:
	
	\begin{mysen}
		Let $\kappa$ be a cardinal. Then $\kappa$ is Mahlo if and only if the collection of ${<}\,\kappa$-models is stationary in $[H(\kappa^+)]^{<\kappa}$.
	\end{mysen}
	
	We now combine the three properties of stationary reflection, diamond and the existence of ${<}\,\kappa$-models in the following definition:
	
	\begin{mydef}
		Let $\kappa$ be a Mahlo cardinal. We write $\RwD(\kappa)$ (read \emph{reflection with diamond at $\kappa$}) if there is a function $l\colon\kappa\to V_{\kappa}$ such that whenever $S\subseteq\kappa$ is stationary and does not contain any inaccessible cardinals and $A,x\in H(\kappa^+)$ there is a $\kappa$-model $M$ with $A,x\in M$ such that $S\cap(M\cap\kappa)$ is stationary in $M\cap\kappa$ and $l(M\cap\kappa)=\pi_M(A)$.
	\end{mydef}
	
	The next Theorem is the main result of this section:
	
	\begin{mysen}\label{ReflectionDiamond}
		Let $\kappa$ be a Mahlo cardinal. There is a forcing extension in which $\kappa$ is Mahlo and $\RwD(\kappa)$ holds.
	\end{mysen}
	
	The idea of the proof is the following (building on \cite{HarringtonShelahExactEquiconsistencies}): We take an iteration $\Add(\kappa)*\dP_{\kappa^+}$ where $\dP_{\kappa^+}$ is the forcing killing all nonreflecting stationary subsets of $\kappa$ that don't contain inaccessible cardinals. In this iteration, the stationarity of a name $\dot{S}$ in $\kappa\cap M$ for some ${<}\,\kappa$-model $M$ will be decided by a condition whose ``Add-part'' is contained in $M$ so that we can arbitrarily choose the value of $l(M\cap\kappa)$ in order to obtain that a specific set is guessed.
	
	\begin{proof}
		Let $(A_{\alpha})_{\alpha\in\kappa^+}$ enumerate all elements of $H(\kappa^+)$ such that any element occurs unboundedly often. We construct an iteration $(\dP_{\alpha},\dot{\dQ}_{\alpha})_{\alpha<\kappa^+}$ with ${<}\,\kappa$-support which will have the following properties for $\alpha\leq\kappa^+$ (we let $\dP_{\kappa^+}$ be the direct limit):
		\begin{enumerate}
			\item[$(1)_{\alpha}$] If $M$ is a $\kappa$-model with $(A_{\alpha})_{\alpha<\kappa^+}\in M$ and $p\in\dP_{\alpha}\cap M$ there exists a condition $p_{\alpha}(M,p)\leq p$ such that whenever $D\in M$ is open dense in $\dP_{\alpha}$ there exists $p'\in D\cap M$ with $p_{\alpha}(M,p)\leq p'$. Moreover, $p_{\alpha}(M,p)(0)\subseteq M$ and for all $\beta<\alpha$, $p_{\alpha}(M,p)\uhr\beta=p_{\beta}(M,p\uhr\beta)$.
			\item[$(2)_{\alpha}$] $\dP_{\alpha}$ is $\kappa$-distributive.
			\item[$(3)_{\alpha}$] $\dP_{\alpha}$ preserves the Mahloness of $\kappa$.
			\item[$(4)_{\alpha}$] (if $\alpha<\kappa$) $\dP_{\alpha}$ has a dense subset of cardinality $\kappa$.
		\end{enumerate}
		
		By standard techniques, $(1)_{\alpha}$ implies both $(2)_{\alpha}$ and $(3)_{\alpha}$ (so it suffices to show $(1)_{\alpha}$ and $(4)_{\alpha}$): Given a $\dP_{\alpha}$-name for a function $\dot{f}$ from an ordinal below $\kappa$ and a club subset $\dot{C}$ of $\kappa$, take any $\kappa$-model containing $\dot{f}$ and $\dot{C}$. Then $p_{\alpha}(M)$ both decides $\dot{f}$ and forces $M\cap\kappa\in\dot{C}$.
		
		Assume $(\dP_{\beta},\dot{\dQ}_{\beta})_{\beta<\alpha}$ has been defined. We first define $\dP_{\alpha}$ and $\dot{\dQ}_{\alpha}$ and then show that they work. If $\alpha=\gamma+1$, let $\dP_{\alpha}:=\dP_{\gamma}*\dot{\dQ}_{\gamma}$. If $\alpha$ is a limit ordinal of cofinality ${<}\,\kappa$, let $\dP_{\alpha}$ be the full limit of the iteration, if $\alpha$ is a limit ordinal of cofinality $\kappa$, let $\dP_{\alpha}$ be the direct limit.
		
		For the $\dot{\dQ}_{\alpha}$: If $\alpha=0$, let $\dot{\dQ}_{\alpha}:=\dot{\Add}(\check{\kappa})$. Otherwise, assume $A_{\alpha}=(\dot{S}_{\alpha},x_{\alpha},\dot{A}_{\alpha},p_{\alpha})$ such that $p_{\alpha}\in\dP_{\alpha}$, $x_{\alpha}\in H(\kappa^+)$ and $\dot{S}_{\alpha}$ is a nice $\dP_{\alpha}$-name for a subset of $\kappa$ consisting of singular ordinals. Assume further that for any $\kappa$-model $M$ containing $(A_{\alpha})_{\alpha<\kappa^+}$ and $\alpha$ (as elements) the condition $p_{\alpha}(M,p_{\alpha})$ does not force $\dot{S}_{\alpha}\cap(M\cap\kappa)$ to be stationary in $(M\cap\kappa)$. Then we let $\dot{\dQ}_{\alpha}$ be forced by $p_{\alpha}$ to be equal to $\CU(\check{\kappa}\smallsetminus\dot{S}_{\alpha})$ (defined as the set of all closed bounded subsets of $\check{\kappa}\smallsetminus\dot{S}_{\alpha}$, ordered by end-extension) and by conditions incompatible with $p_{\alpha}$ to be the trivial poset. If the assumption does not hold, every condition forces $\dot{\dQ}_{\alpha}$ to be trivial.
		
		Now we show that $(1)_1$ holds: Let $M$ be any $\kappa$-model and $p\in \dP_1\cap M=\Add(\kappa)\cap M$. Let $(D_{\delta})_{\delta<M\cap\kappa}$ enumerate all open dense subsets of $\dP_1$ in $M$ and let $(p_{\delta})_{\delta<M\cap\kappa}$ be a descending sequence with $p_{\delta}\in D_{\delta}\cap M$ and $p_1\leq p$. Lastly, let $p_1(M,p):=\bigcup_{\delta<M\cap\kappa}p_{\delta}$ which clearly is as required. We note for later that $p_1(M,p)\subseteq M$. Clearly $(4)_1$ holds, since $\dP_1=\Add(\kappa)$ itself has cardinality $\kappa$.
		
		Now assume $(1)_{\gamma}$ holds and $\alpha=\gamma+1$. Let $M$ be any $\kappa$-model with $(A_{\alpha})_{\alpha<\kappa^+}\in M$ and $p\in\dP_{\alpha}\cap M$. If $\alpha\notin M$, we are done: Whenever $D\in M$ is open dense in $\dP_{\alpha}$, $\sup(\{\supp(p)\;|\;p\in D\})\in M$ and thus below $\alpha$. So $p_{\gamma}(M,p)$ is as required (since $p\in\dP_{\gamma}$ by the same reason). Thus assume $\alpha\in M$, so $\gamma\in M$. Let $\nu:=M\cap\kappa$ and let $G_{\gamma}$ be any $\dP_{\gamma}$-generic filter containing $p_{\gamma}(M,p\uhr\gamma)$. First assume $\dot{\dQ}_{\gamma}^{G_{\gamma}}$ is nontrivial.
		
		\begin{myclaim}
			In $V$, $p_{\gamma}(M,p\uhr\gamma)$ forces over $\dP_{\gamma}$ that $\dot{S}_{\gamma}\cap \nu$ is nonstationary in $\nu$.
		\end{myclaim}
		
		\begin{proof}
			Let $\alpha\in\nu$. Then the set of all conditions deciding $\check{\alpha}\in\dot{S}_{\gamma}$ is open dense in $\dP_{\gamma}$. Ergo $p_{\gamma}(M,p\uhr\gamma)$ decides $\check{\alpha}\in\dot{S}_{\gamma}$. By $(2)_{\gamma}$, $\dot{S}_{\gamma}\cap\nu$ is forced to be stationary if and only if $\{\delta<\nu\;|\;p_{\gamma}(M,p\uhr\gamma)\Vdash\check{\delta}\in\dot{S}_{\gamma}\}$ is actually stationary (in $V$), so by assumption, $\dot{S}_{\gamma}\cap\nu$ is forced to be nonstationary.
		\end{proof}
		
		Now work in $V[G_{\gamma}]$.
		
		\begin{myclaim}
			$M[G_{\gamma}]\cap V=M$ and $M[G_{\gamma}]$ is a $\kappa$-model.
		\end{myclaim}
		
		\begin{proof}
			Let $\tau\in M$ with $\tau^{G_{\gamma}}\in V$. Ergo in $M$ the set of all $q$ forcing $\tau=\check{x}_q$ for some $x_q\in V$ is open dense. By $(1)_{\gamma}$, there is such a $q$ in $M\cap \dP_{\gamma}$ which is above $p(M,p\uhr\gamma)$. The corresponding $x_q$ is in $M$ as well by elementarity. Lastly, $\tau^{G_{\gamma}}=x_q$, since $q\geq p(M,p\uhr\gamma)\in G_{\gamma}$
			
			Clearly $M[G_{\gamma}]$ is closed under ${<}\,\kappa$-sequences of ordinals (by $(2)_{\gamma}$) and thus under ${<}\,\kappa$-sequences.
		\end{proof}
		
		So in $V[G_{\gamma}]$, $\dot{S}_{\gamma}^{G_{\gamma}}\cap\nu$ is disjoint from some $C\subseteq\nu$ which is club in $\nu$. Define
		
		$$T_{\gamma}:=\{s\in\dot{\dQ}_{\gamma}^{G_{\gamma}}\cap M[G_{\gamma}]\;|\;\max(s)\in C\}$$
		
		\begin{myclaim}
			$T_{\gamma}$ is dense in $\dot{\dQ}_{\gamma}^{G_{\gamma}}\cap M[G_{\gamma}]$ and ${<}\,\nu$-closed.
		\end{myclaim}
		
		\begin{proof}
			If $s\in\dot{\dQ}_{\gamma}^{G_{\gamma}}\cap M[G_{\gamma}]$, $\max(s)<\nu$, so there is $\delta\in C$ with $\max(s)<\delta$. Then $s\cup\{\delta\}$ is as required. ${<}\,\nu$-closure is clear as $C$ is club and disjoint from $\dot{S}_{\gamma}^{G_{\gamma}}\cap\nu$.
		\end{proof}
		
		So in $V[G_{\gamma}]$ we can let $(D_{\delta})_{\delta<\nu}$ enumerate all open dense subsets of $\dot{\dQ}_{\gamma}^{G_{\gamma}}$ lying in $M[G_{\gamma}]$ and find a descending sequence $(q_{\delta})_{\delta<\nu}$ with $q_0\leq p(\gamma)$ and $q_{\delta}\in D_{\delta}$. Lastly, let $q_{\gamma}:=\bigcup_{\delta<\nu}p_{\delta}\cup\{\nu\}$ which is in $\dot{\dQ}_{\gamma}^{G_{\gamma}}$ (because $\dot{S}_{\gamma}^{G_{\gamma}}$ contains no inaccessible cardinals). If $\dot{\dQ}_{\gamma}^{G_{\gamma}}$ is trivial, we just take $q_{\gamma}:=\emptyset$. Then, whenever $D\in M[G_{\gamma}]$ is open dense in $\dot{\dQ}_{\gamma}^{G_{\gamma}}$, there is $q\in D\cap M$ with $q_{\gamma}\leq q$. By the maximum principle, there is a $\dP_{\gamma}$-name $\dot{q}_{\gamma}$ forced to have these properties. It follows that $p_{\alpha}(M,p):=(p_{\gamma}(M,p\uhr\gamma))^{\frown}\dot{q}_{\gamma}$ witnesses $(1)_{\alpha}$ for $M$ and $p$.
		
		Now we prove $(4)_{\alpha}$: Let $\dR_{\gamma}$ be a dense subset of $\dP_{\gamma}$ of cardinality $\kappa$. Let $\dR_{\alpha}$ consist of those $p\in\dP_{\alpha}$ such that $p\uhr\gamma\in\dR_{\gamma}$ and $p\uhr\gamma\Vdash p(\gamma)=\check{c}$ for some $c\in V$. It follows from $(1)_{\alpha}$ that $\dR_{\alpha}$ is dense in $\dP_{\alpha}$. Clearly $|\dR_{\gamma}|=\kappa$.
		
		Now assume $\alpha$ is a limit and $(1)_{\gamma}$ holds for all $\gamma<\alpha$. If $M\cap\alpha$ is unbounded in $\alpha$, simply let $p_{\alpha}(M,p)$ be the indirect limit of $p_{\gamma}(M,p\uhr\gamma)$ for $\gamma<\alpha$ (since that implies $\cf(\alpha)<\kappa$). Otherwise, $\gamma:=\sup(M\cap\alpha)<\alpha$, so we can let $p_{\alpha}(M,p):=p_{\gamma}(M,p\uhr\gamma)$. Lastly, $(4)_{\alpha}$ holds since we are taking limits with ${<}\,\kappa$-support (and $\kappa^{<\kappa}=\kappa$ by inaccessibility).
		
		Now we show that $\dP_{\kappa^+}$ forces $\RwD(\kappa)$. In $V$, let $f\colon\kappa\to V_{\kappa}$ be any bijection. In $V[G]$, let $G_1$ be the $\dP_1=\Add(\kappa)$-generic filter induced by $G$, let $g:=\bigcup G_1$ and let $l(\nu):=f(g(\nu))$ for any $\nu<\kappa$.
		
		Now assume $\dot{S},\dot{A},\dot{x}$ are (in the first case, nice) $\dP_{\kappa^+}$-names for a subset of $\kappa$ and elements of $H(\kappa^+)$ respectively, forced by some $p\in\dP_{\kappa^+}$. We assume for simplicity that $\dot{A}\in\dot{x}$, otherwise we can replace $\dot{x}$ be $\{\dot{x},\dot{A}\}$. By the $\kappa^+$-cc., we can assume they are already $\dP_{\alpha}$-names for some $\alpha<\kappa^+$ and since $\dP_{\alpha}$ has a dense subset of cardinality $\kappa$ we can assume that $\dot{A},\dot{x}$ are elements of $H(\kappa^+)$. Assume $A_{\alpha}=(\dot{S},\dot{x},\dot{A},p)$. We check two cases:
		
		\begin{itemize}
			\item[Case 1:] There exists a $\kappa$-model $M$ containing $(A_{\alpha})_{\alpha<\kappa^+}$ and $\alpha$ (as elements) such that $p_{\alpha}(M,p)$ forces $\dot{S}\cap(M\cap\kappa)$ to be stationary in $M\cap\kappa$. We have $p_{\alpha}(M,p)\uhr 1=p_1(M,p\uhr 1)\subseteq M$. Furthermore, $\dot{S},\dot{A},\dot{x},p\in M$ by elementarity (so $p_{\alpha}(M,p)$ actually exists). Just like previously, we have:
			
			\begin{myclaim}
				There is $A_M$ such that $p_{\alpha}(M,p)\Vdash\check{A}_M=\dot{A}\cap M[\Gamma]$.
			\end{myclaim}
			
			Additionally, $\pi_M[A_M]\in V$ and is equal to $f(\delta)$ for some $\delta$. Let $p_M$ be the extension of $p_{\kappa^+}(M,p)$ with $p(0)=\delta$. We claim that $p_M$ (which extends $p$) forces $M[\Gamma]$ to witness $\RwD(\kappa)$ for $\dot{S},\dot{A}$ and $\dot{x}$. To this end, let $G$ be $\dP_{\kappa^+}$-generic containing $p_M$. Then $M[G]$ is a $\kappa$-model. Furthermore, $p_M\leq p_{\alpha}(M,p)$, so $\dot{S}^G\cap(M\cap\kappa)=\dot{S}^G\cap(M[G]\cap\kappa)$ is stationary in $(M[G]\cap\kappa)$. Clearly, $\dot{x}^G,\dot{A}^G\in M$. Lastly,
			$$\pi(\dot{A}^G)=\pi[\dot{A}^G\cap M[G]]=\pi[A_M]=f(\delta)=f(g(M\cap\kappa))$$
			
			\item[Case 2:] Case 1 fails. In this case, $p$ forces $\dot{\dQ}_{\alpha}=\CU(\check{\kappa}\smallsetminus\dot{S}_{\alpha})$. However, then $p$ forces $\dot{S}$ to be nonstationary.
		\end{itemize}
		
		Ergo $p$ forces that if $\dot{S}$ is stationary, there exists a $\kappa$-model as required.
	\end{proof}
	
	The main crux now of the argument is that certain $\kappa$-cc.\ posets preserve stationary reflection in some way (noting that they can never preserve full stationary reflection if they collapse $\kappa$ to be a successor), which we prove with an argument similar to Usuba in \cite[Proposition 5.1]{UsubaPartialRefl}.  We call a poset $\dP$ \emph{nicely $\kappa$-cc.\@} if $\dP$ is $\kappa$-cc.\@, $\dP\in H(\kappa^+)$ and $\dP\cap M$ is $|M|$-cc.\ for every $\kappa$-model $M$ with $\dP\in M$.
	
	\begin{mylem}\label{ReflectionNicelyCC}
		Assume $\kappa$ is Mahlo and $\RwD(\kappa)$ holds, witnessed by $l\colon\kappa\to V_{\kappa}$. Let $\dR$ be a poset with the following properties:
		\begin{enumerate}
			\item $\dR$ is nicely $\kappa$-cc.
			\item $\dR$ does not change the cofinality of any inaccessible cardinal to $\omega$.
		\end{enumerate}
		Let $\dot{S}$ be an $\dR$-name for a stationary subset of $E_{\omega}^{\kappa}$ and $A,x\in H(\kappa^+)$. Then there is a $\kappa$-model $M$ and $r\in\dR\cap M$ such that $x\in M$, $l(M\cap\kappa)=\pi_M(A)$ and
		$$r\Vdash_{\dR\cap M}\dot{S}\cap(M\cap\kappa)\text{ is stationary in }(M\cap\kappa)$$
	\end{mylem}
	
	\begin{mybem}
		We note that $\dot{S}\cap(M\cap\kappa)$ is only forced to be stationary over $\dR\cap M$, not over the whole forcing $\dR$. It is possible that the whole forcing $\dR$ might again force $\dot{S}\cap(M\cap\kappa)$ to be nonstationary. However, e.g. in the case that $\dR=\Coll(\omega_1,{<}\,\kappa)$, this is not the case as the quotient forcing $\dR/(\dR\cap M)$ is countably closed.
	\end{mybem}
	
	\begin{proof}
		Assume the parameters are given. Let
		$$T:=\{\delta\in\kappa\;|\;\exists r\in\dR(r\Vdash\check{\delta}\in\dot{S})\}$$
		Then $T$ is stationary in $\kappa$ and does not contain any inaccessible cardinals. By $\RwD(\kappa)$ there is a $\kappa$-model $M$ with $A,x,T\in M$ such that $T\cap(M\cap\kappa)$ is stationary in $(M\cap\kappa)$ and $l(M\cap\kappa)=\pi_M(A)$.
		
		Assume toward a contradiction that there is no $r\in\dR\cap M$ that forces $\dot{S}\cap(M\cap\kappa)$ to be stationary in $M\cap\kappa$, i.e. $1_{\dR\cap M}$ forces the negation. Because $\dR\cap M$ is $(M\cap\kappa)$-cc. there exists a club $C\subseteq M\cap\kappa$ such that $1_{\dR\cap M}$ forces $\check{C}\cap\dot{S}$ to be empty. However, there is an ordinal $\gamma\in T\cap C$ and therefore a condition $r\in\dR$ (which is also in $M$ by elementarity since $\gamma\in M$) that forces $\check{\gamma}\in\dot{S}\cap\check{C}$, a contradiction.
	\end{proof}
	
	We can now prove Theorem \ref{ShelahTheorem}:
	
	\begin{proof}
		Let $\kappa$ be a Mahlo cardinal. By Theorem \ref{ReflectionDiamond} we can assume that $\RwD$ holds. We now let $\dP:=(\dP_{\alpha},\dot{\dQ}_{\alpha})_{\alpha<\kappa}$ be an RCS iteration such that $\dot{\dQ}_{\alpha}:=\dP(l(\alpha))$ whenever $\alpha$ is inaccessible and $l(\alpha)$ is a $\dP_{\alpha}$-name for a stationary subset of $E_{\omega}^{\alpha}$ and $\dot{\dQ}_{\alpha}:=\dot{\Coll}(\omega_1,|\dP_{\alpha}|^{++})$ otherwise (the collapses ensure that our iterands are well-separated).
		
		We show that $F^+(\omega_2)$ holds in the model obtained by forcing with $\dP$. We have already seen that $\dP$ has the $S$-condition for some set $S$ of regular cardinals above $\omega_2$. Ergo $\dP$ preserves $\omega_1$ (and in fact does not add reals if we assume $\CH$ in the ground model). As $\dP$ has the $\kappa$-cc.\@, $\kappa$ remains a cardinal and is collapsed to $\omega_2$.
		
		Now let $G$ be $\dP$-generic and let $\dot{A}$ be a $\dP=\dP_{\kappa}$-name for a stationary subset of $E_{\omega}^{\kappa}$. By Lemma \ref{ReflectionNicelyCC} there is a $\kappa$-model $M\prec H(\Theta)$ containing all relevant parameters such that $l(M\cap\kappa)=\pi(\dot{A})$ and $p\in\dP\cap M$ forcing that $\dot{A}\cap(M\cap\kappa)=\pi(\dot{A})$ is stationary in $M\cap\kappa$. By elementarity we have that $\pi(\dot{A})$ is a $\dP_{M\cap\kappa}$-name for a stationary subset of $E_{\omega}^{M\cap\kappa}$. Ergo $\dP_{(M\cap\kappa)+1}=\dP_{M\cap\kappa}*\dP(\pi(\dot{A}))$. It follows that in $V[G]$ there exists a closed set $c\subseteq\pi(\dot{A})^{G\cap\dP_{M\cap\kappa}}$ with ordertype $\omega_1$. As $\pi(\dot{A})^{G\cap\dP_{M\cap\kappa}}=\dot{A}^G\cap(M\cap\kappa)$, there exists a closed set $c\subseteq\dot{A}^G$ with ordertype $\omega_1$.
	\end{proof}
	
	We now turn to the consistency of the strongest principle. As stated before, to show $F^+((D_i)_{i\in\omega_1},\kappa)$ for an arbitrary partition $(D_i)_{i\in\omega_1}$ of $\omega_1$, it suffices to show that given any sequence $(A_i)_{i<\omega_1}$ of stationary subsets of $E_{\omega}^{\kappa}$ there is a normal function $f\colon\omega_1\to\kappa$ such that $f(i)\in A_i$ for any $i\in\omega_1$. This is straightforwardly added by the following poset: Let $\dP((A_i)_{i<\omega_1})$ consist of normal functions $p\colon\alpha+1\to\kappa$ such that $\alpha<\omega_1$ and $p(i)\in A_i$ for any $i<\alpha+1$, ordered by end-extension.
	
	The following was shown by Shelah:
	
	\begin{mylem}
		$\dP((A_i)_{i\in\omega_1})$ has the $\{\aleph_2\}$-condition.
	\end{mylem}
	
	This can be used to show that from a weakly compact cardinal we can obtain a forcing extension where $F^+((D_i)_{i\in\omega_1},\kappa)$ holds for any partition $(D_i)_{i\in\omega_1}$ of $\omega_1$.
	
	\section{Adding Witnesses to the Failure of Friedman Properties}
	
	In this section we introduce posets which add witnesses to the failure of various instances of $F$ or $F^+$ and demonstrate their properties.
	
	\begin{mydef}
		Let $\kappa\geq\omega_2$ be a regular cardinal.
		
		\begin{enumerate}
			\item Let $\lambda\leq\kappa$ be any cardinal. $\dP_{F(\lambda)}(\kappa)$ consists of regressive functions $p\colon\delta\to\lambda$ for some ordinal $\delta<\kappa$ such that there is no normal function $g\colon\omega_1\to\delta$ such that $p\circ g$ is constant.
			\item Let $(D_i)_{i\in\omega_1}$ be a partition of $\omega_1$. $\dP_{F^+((D_i)_{i\in\omega_1})}(\kappa)$ consists of functions $p\colon\delta\times\omega_1\to 2$ for some ordinal $\delta<\kappa$ such that the following holds:
			\begin{enumerate}
				\item For any $\alpha\in\delta$ there is at most one $i\in\omega_1$ such that $p(\alpha,i)=1$.
				\item For any $\alpha\in\delta$ and $i\in\omega_1$, if $\cf(\alpha)\neq\omega$, $p(\alpha,i)=0$.
				\item There is no normal function $g\colon\omega_1\to\delta$ such that for all $i\in\omega_1$, $p[g[D_i]\times\{i\}]=\{1\}$.
			\end{enumerate}
		\end{enumerate}
		
		We order both posets by end-extension.
	\end{mydef}
	
	\begin{mybem}
		The ``disjointness condition'' (1) occurs in the definition of $\dP_{F^+((D_i)_{i\in\omega_1})}(\kappa)$ for the following reason: Without it, given a $\dP_{F^+((D_i)_{i\in\omega_1})}(\kappa)$-generic filter $G$ and $F:=\bigcup G$, the set $A:=\{\alpha\in\kappa\;|\;\forall i\in\omega_1(F(\alpha,i)=1)\}$ would be a stationary subset of $E_{\omega}^{\kappa}$ not containing any closed copy of $\omega_1$, so actually the weakest form of $F^+((E_i)_{i\in\omega_1},\kappa)$ would fail in the extension and the poset would be ill-suited to force a distinction between variants of $F^+$.
	\end{mybem}
	
	We now show that both posets do not collapse cardinals (at least assuming $\GCH$). As both of them have size $\kappa^{<\kappa}$, it suffices to show that they do not add new sequences of ordinals of length ${<}\,\kappa$. In order to do so, we show that they are strategically closed. We start with the easier case which is $\dP_{F^+((D_i)_{i\in\omega_1})}(\kappa)$.
	
	\begin{mylem}
		Let $\kappa\geq\omega_2$ be a regular cardinal and $(D_i)_{i\in\omega_1}$ any partition of $\omega_1$. $\dP_{F^+((D_i)_{i\in\omega_1})}(\kappa)$ is $\kappa$-strategically closed.
	\end{mylem}
	
	\begin{proof}
		We construct our strategy in such a way that, at limits of uncountable cofinality, there is a club on which our conditions take value $0$ no matter the second argument.
		
		Assume we are given a sequence $(p_{\alpha})_{\alpha<\gamma}$ for $\gamma<\kappa$. We let $p_{\gamma}':=\bigcup_{\alpha<\gamma}p_{\alpha}$ and $p_{\gamma}$ a function with domain $(\dom(p_{\gamma}')+1)\times\omega_1$, where $p_{\gamma}(\dom(p_{\gamma}'),i)=0$ for any $i<\omega_1$. We now show that this strategy works. Assume toward a contradiction that $g$ is a normal function from $\omega_1$ to $\dom(p_{\gamma})$ such that $p_{\gamma}[g[D_i]\times\{i\}]=\{1\}$ for any $i\in\omega_1$. It follows that $g$ actually maps to $\dom(p_{\gamma}')$. Moreover, $\im(g)$ is unbounded in $\dom(p_{\gamma}')$ (therefore it is club), since all the $p_{\alpha}$ for $\alpha<\gamma$ are conditions. Furthermore, by the construction, the set $\{\dom(p_{\alpha}')\;|\;\alpha<\gamma\text{ is a limit}\}$ is club in $\dom(p_{\gamma}')$ as well, so there is some $\beta\in\omega_1$, $\beta\in D_i$ for some $i\in\omega_1$, such that $g(\beta)=\dom(p_{\alpha}')$ for some limit $\alpha<\gamma$. However, in that case we have $p_{\gamma}(\dom(p_{\alpha}'),i)=p_{\alpha}(\dom(p_{\alpha}'),i)=0$, a contradiction.
	\end{proof}
	
	In the case of $F(\lambda,\kappa)$, we cannot expect such a strong result. It is easy to see that the proof of the previous Lemma (or its straighforward modification) would fail for $\dP_{F(\lambda)}(\kappa)$ since we also would need to remove the possibility that $p_{\gamma}$ is constantly $0$ on a closed copy of $\omega_1$. Additionally, $F$ reflects downwards, so assuming $F(\lambda,\kappa)$ holds we cannot force $F(\lambda,\kappa^+)$ to fail without adding a $\kappa$-length sequence of ordinals. However, this is the only obstacle:
	
	\begin{mylem}
		Let $\kappa$ be regular and $\lambda\leq\kappa$ any cardinal. Let $\delta<\kappa$ be regular. The following are equivalent:
		\begin{enumerate}
			\item $\dP_{F(\lambda)}(\kappa)$ is $\delta+1$-strategically closed.
			\item $\dP_{F((\lambda)}(\kappa)$ is ${<}\,\delta^+$-distributive.
			\item $F(\lambda,\delta)$ fails.
		\end{enumerate}
	\end{mylem}
	
	\begin{mybem}
		We abuse notation slightly in two ways: We consider the principle $F(\lambda,\delta)$ for $\lambda>\delta$ (in which case it is clearly equivalent to $F(\delta,\delta)$) and also for $\delta=\omega_1$ (in which case it clearly always fails).
	\end{mybem}
	
	\begin{proof}
		It is clear that (1) implies (2). Furthermore, (2) is easily seen to imply (3): Assuming $F(\lambda,\delta)$ holds there is no condition $p\in\dP_{F(\lambda)}(\kappa)$ such that $\dom(p)\geq\delta$, so INC has an easy winning strategy in $G(\dP,\delta+1)$ by simply increasing the domain of the previously played condition.
		
		Now assume $F(\lambda,\delta)$ fails as witnessed by a regressive function $f\colon\delta\to\lambda$. We will construct a winning strategy $\sigma$ for COM in $G(\dP,\delta+1)$. Assume $(p_{\alpha})_{\alpha<\gamma}$ is a run of the game where $\gamma\leq\delta$ is even. If $\gamma$ is not a limit, we simply let $p_{\gamma}$ be an arbitrary extension of $p_{\beta}$ with $\dom(p_{\gamma})\geq\gamma$ (by assuming inductively that $\dom(p_{\gamma}-2)\geq\gamma-2$, this is clearly possible without extending the domain by more than $2$ elements so we can do this easily without leaving $\dP_{F(\lambda)}(\kappa)$). Now assume $\gamma$ is a limit. Let $p_{\gamma}':=\bigcup_{\alpha<\gamma}p_{\alpha}$ and $p_{\gamma}$ an extension of $p_{\gamma}'$ with domain $\dom(p_{\gamma}')+1$ such that $p_{\gamma}(\dom(p_{\gamma}'))=f(\gamma)$ (if $\gamma<\delta$, otherwise arbitrary). We will show that this works as intended. First of all, $p_{\gamma}$ is regressive because $\dom(p_{\gamma}')\geq\gamma$ and so
		$$p_{\gamma}(\dom(p_{\gamma}'))=f(\gamma)<\gamma\leq\dom(p_{\gamma}')$$
		Secondly, assume $g\colon\omega_1\to\dom(p_{\gamma})$ is a normal function such that $p_{\gamma}\circ g$ is constant. As in the previous proof, $\im(g)$ is club in $\dom(p_{\gamma}')$. The set $c:=\{\dom(p_{\alpha}')\;|\;\alpha<\gamma\text{ is a limit}\}$ is club in $\dom(p_{\gamma}')$ as well and therefore so is $c\cap\im(g)=:d$. It follows that there is a normal and surjective function $h\colon\omega_1\to d$. For any $\alpha<\omega_1$, let $i(\alpha)$ be such that $h(\alpha)=\dom(p_{i(\alpha)}')$. Then clearly $i$ is a normal function from $\omega_1$ into $\gamma$. Moreover, for any $\alpha,\alpha'\in\omega_1$, we have
		$$f(i(\alpha))=p_{\gamma}(\dom(p_{i(\alpha)}'))=p_{\gamma}(h(\alpha))=p_{\gamma}(h(\alpha'))=p_{\gamma}(\dom(p_{i(\alpha)}'))=f(i(\alpha'))$$
		contradicting the fact that $f$ witnesses the failure of $F(\lambda,\delta)$.
	\end{proof}
	
	The main difference between $\dP_{F(\lambda)}(\kappa)$ (or $\dP_{F^+((D_i)_{i\in\omega_1})}(\kappa)$) and the poset adding a nonreflecting stationary set lies in the following observation:
	
	\begin{myobs}\label{MainObs}
		If $G\subseteq\dP_{F(\lambda)}(\kappa)$ (or $G\subseteq\dP_{F^+((D_i)_{i\in\omega_1})}(\kappa)$) is a generic filter, $\bigcup G$ is a condition in $(\dP_{F(\lambda)}(\kappa^+))^{V[G]}$ (or $(\dP_{F^+((D_i)_{i\in\omega_1})}(\kappa^+))^{V[G]}$).
	\end{myobs}
	
	This is what will later allow us to construct master conditions and lift ground-model embeddings.
	
	Now we show that $\dP_{F(\lambda)}$ and $\dP_{F^+((D_i)_{i\in\omega_1})}$ work as intended. For $\dP_{F(\lambda)}$ this is easy:
	
	\begin{mylem}
		Let $\kappa$ be regular and $\lambda\leq\kappa$ any cardinal. Assume $F(\lambda,\delta)$ fails for any regular cardinal $\delta<\kappa$. Then $\dP_{F(\lambda)}(\kappa)$ adds a function $f\colon\kappa\to\lambda$ such that there is no normal function $g\colon\omega_1\to\kappa$ with $f\circ g$ constant.
	\end{mylem}
	
	\begin{proof}
		Because $F(\lambda,\delta)$ fails for any regular cardinal $\delta<\kappa$, for any $\alpha<\kappa$, the set of all conditions $p\in\dP_{F(\lambda)}(\kappa)$ with $\alpha\in\dom(p)$ is dense in $\dP_{F(\lambda)}(\kappa)$. Consequently, if $G$ is $\dP_{F(\lambda)}(\kappa)$-generic, $\bigcup G$ is a function from $\kappa$ into $\lambda$ and clearly there is no normal function $g\colon\omega_1\to\kappa$ such that $f\circ g$ is constant because any such $g$ would necessarily be bounded in $\kappa$ and even in the ground model.
	\end{proof}
	
	For $\dP_{F^+((D_i)_{i\in\omega_1})}$ we use the well-known idea that ``most'' generically added sets are stationary:
	
	\begin{mylem}
		Let $\kappa$ be regular and $(D_i)_{i\in\omega_1}$ a partition of $\omega_1$. $\dP_{F^+((D_i)_{i\in\omega_1})}(\kappa)$ adds a disjoint sequence $(A_i)_{i\in\omega_1}$ of stationary subsets of $E_{\omega}^{\kappa}$ such that there is no normal function $g\colon\omega_1\to\kappa$ with $g[D_i]\subseteq A_i$ for all $i\in\omega_1$.
	\end{mylem}
	
	\begin{proof}
		Let $G$ be a $\dP_{F^+((D_i)_{i\in\omega_1})}(\kappa)$-generic filter and let $F:=\bigcup G$. For any $i\in\omega_1$, let $A_i:=\{\alpha\in\kappa\;|\;F(\alpha,i)=1\}$. By the definition, $A_i\subseteq E_{\omega}^{\kappa}$. Moreover, there clearly does not exist a function $g$ as stated. The only part left to show is that any $A_i$ is stationary in $\kappa$. To this end, let $\dot{C}$ be a $\dP_{F^+((D_i)_{i\in\omega_1})}(\kappa)$-name for a club subset of $\kappa$, as forced by some condition $p\in\dP_{F^+((D_i)_{i\in\omega_1})}(\kappa)$.
		
		We now construct sequences $(p_n)_{n\in\omega}$ and $(\gamma_n)_{n\in\omega}$ as follows: Let $(p_0,\gamma_0)$ be any pair such that $p_0\leq p$ and $p_0\Vdash\check{\gamma}_0\in\dot{C}$. Given $(p_n,\gamma_n)$, let $(p_{n+1},\gamma_{n+1})$ be any pair with $\dom(p_{n+1})=\delta\times\omega_1$ for some $\delta>\gamma_n$ and $\gamma_{n+1}>\dom(p_n)$ such that $p_{n+1}\Vdash\check{\gamma}_{n+1}\in\dot{C}$. Lastly, let $p:=\bigcup_{n\in\omega}p_n$ and $\gamma:=\bigcup_{n\in\omega}\gamma_n$. It follows by closure of $\dot{C}$ that $p\Vdash\check{\gamma}\in\dot{C}$. Moreover, $\dom(p)=\gamma\times\omega_1$. As such, we can let $p'$ be an extension of $p$ such that $\dom(p')=(\gamma+1)\times\omega_1$ and $p'(\gamma,i)=1$, $p'(\gamma,j)=0$ for $j<\omega_1, j\neq i$. Because $\cf(\gamma)=\omega$, $p'\in\dP_{F^+((D_i)_{i\in\omega_1})}(\kappa)$ and $p'$ clearly forces that $\dot{C}\cap\dot{A}_i$ is nonempty.
	\end{proof}
	
	\subsection*{Propagation of $\neg F(\delta)$}\hfill
	
	In this subsection we use the properties of the forcing notion $\dP_{F(\lambda)}(\kappa)$ and well-known statements regarding strategic closure in order to show that the existence of witnesses to $\neg F(\lambda,\kappa)$ can be lifted to arbitrary ordinals below $\kappa^+$ and to show that the failure of $F(\lambda,\delta)$ is compact at singular cardinals.
	
	\begin{mylem}\label{LiftingStratClos}
		Assume $\delta$ is a singular limit ordinal and $\dP$ is a forcing order which is ${<}\,\delta$-strategically closed. Then $\dP$ is $\delta+1$-strategically closed.
	\end{mylem}
	
	\begin{proof}
		By the assumption of the lemma, $\dP$ is in particular $\cf(\delta)+1$-strategically closed. Fix a winning strategy $\sigma$ for COM in $G(\dP,\cf(\delta)+1)$ and also a winning strategy $\sigma_{\alpha}$ for COM in $G(\dP,\alpha+1)$ for any $\alpha<\delta$. For simplicity, we will assume that any such strategy maps $(1_{\dP})_{\beta<\gamma}$ to $1_{\dP}$. We will construct a winning strategy $\sigma_{\delta}$ for COM in $G(\dP,\delta+1)$. To this end, fix a normal function $f\colon\cf(\delta)\to\delta$ that maps successor ordinals to successor ordinals.
		
		Let $(p_{\beta})_{\beta<\gamma}$ be given such that $\gamma<\lambda$ is even (so $\gamma$ may be a limit ordinal).
		
		As a first case, assume $\gamma\notin\im(f)$. Let $\alpha$ be minimal such that $\gamma<f(\alpha+1)$. Then we let
		$$\sigma_{\delta}((p_{\beta})_{\beta<\gamma}):=\sigma_{f(\alpha+1)}((1_{\dP})_{\beta\leq f(\alpha)}{}^{\frown}(p_{\beta})_{\beta\in(f(\alpha),\gamma)})$$
		This is possible because $(1_{\dP})_{\beta<f(\alpha)}{}^{\frown}(p_{\beta})_{\beta\in(f(\alpha),\gamma)}$ has been played according to $\sigma_{f(\alpha+1)}$. Additionally, $\sigma_{\delta}((p_{\beta})_{\beta<\gamma})$ is a lower bound of $(p_{\beta})_{\beta\in(f(\alpha),\gamma)}$ and thus of $(p_{\beta})_{\beta<\gamma}$ as required.
		
		Now assume $\gamma=f(\alpha)$ for some $\alpha$. If $\alpha=\xi+1$, we can find $\mu$ such that $\gamma=\mu+1$ and we let
		$$\sigma_{\delta}((p_{\beta})_{\beta<\gamma}):=\sigma((p_{f(\beta)})_{\beta<\xi}{}^{\frown}p_{\mu})$$
		this is a lower bound of $(p_{\beta})_{\beta<\gamma}$ because $\gamma=\mu+1$.
		
		If $\alpha$ is a limit ordinal, we let
		$$\sigma_{\delta}((p_{\beta})_{\beta<f(\alpha)}):=\sigma((p_{f(\beta)})_{\beta<\alpha})$$
		this works because $(p_{f(\beta)})_{\beta<\alpha}$ has been played according to $\sigma$.
		
		It is clear that this strategy works to produce lower bounds for any $\gamma<\delta+1$.
	\end{proof}
	
	We obtain an easy corollary which is a well-known folklore result.
	
	\begin{mycol}\label{ColStratClos}
		Let $\dP$ be a forcing order and $\mu$ a cardinal. If $\dP$ is $\mu+1$-strategically closed, $\dP$ is ${<}\,\mu^+$-strategically closed.
	\end{mycol}
	
	\begin{proof}
		One shows easily by induction that for any $\delta$, $\dP$ is $\delta+1$-strategically closed. The successor step is clear, in limits we use that any $\delta\in(\mu,\mu^+)$ is singular.
	\end{proof}
	
	By relating the strategic closure of $\dP_{F(\lambda)}(\kappa)$ to the failure of $F(\lambda,\kappa)$, we see the following:
	
	\begin{mylem}\label{FriedmanPumpUp}
		Let $\kappa\geq\omega_2$ be regular and $\lambda\leq\kappa$ any cardinal. Assume $F(\lambda,\kappa)$ fails and $\delta\in[\kappa,\kappa^+)$.
		\begin{enumerate}
			\item If $\lambda<\kappa$ there is a regressive function $g\colon\delta\to\lambda$ that is not constant on any closed copy of $\omega_1$.
			\item If $\lambda=\kappa$ there is a regressive function $g\colon\delta\to\delta$ that is not constant on any closed copy of $\omega_1$.
		\end{enumerate}
	\end{mylem}
	
	\begin{proof}
		Assume $F(\lambda,\kappa)$ fails. Then $\dP_{F(\lambda)}(\kappa^+)$ is $\kappa+1$-strategically closed. By corollary \ref{ColStratClos}, $\dP_{F(\lambda)}(\kappa^+)$ is $\delta+1$-strategically closed. Ergo $\dP_{F(\lambda)}(\kappa^+)$ contains a condition with domain $\delta$ which is as required.
	\end{proof}
	
	A similar proof shows that the negation of $F$ is compact at singular cardinals:
	
	\begin{mylem}
		Let $\kappa$ be singular and $\lambda\leq\kappa$ any cardinal. Assume $F(\lambda,\delta)$ fails for all $\delta<\kappa$.
		\begin{enumerate}
			\item If $\lambda<\kappa$, there is a regressive function $g\colon\kappa\to\lambda$ that is not constant on any closed copy of $\omega_1$.
			\item If $\lambda=\kappa$, there is a regressive function $g\colon\kappa\to\kappa$ that is not constant on any closed copy of $\omega_1$.
		\end{enumerate}
	\end{mylem}
	
	\begin{proof}
		Assuming $F(\lambda,\delta)$ fails for all $\delta<\kappa$, the poset $\dP_{F(\lambda)}(\kappa^+)$ is ${<}\,\kappa$-strategically closed. Ergo it is actually $\kappa+1$-strategically closed. It follows that functions as required exist.
	\end{proof}
	
	\section{Maximal Versions of $\MM$ with Failure of Friedman Properties}
	
	The forcing axiom Martin's Maximum ($\MM$) was introduced by Foreman, Magidor and Shelah in \cite{ForemanMagidorShelahMM}. It is shown in the same paper that $\MM$ implies $F^+((D_i)_{i\in\omega_1},\kappa)$ for any regular $\kappa\geq\omega_2$, where $(D_i)_{i\in\omega_1}$ is the partition given by $D_0=\omega_1$, $D_i=\emptyset$ for $i\neq 0$ (we will later show that $\SRP$, a consequence of $\MM$, implies $F^+((D_i)_{i\in\omega_1},\kappa)$ for any partition $(D_i)_{i\in\omega_1}$ and any regular $\kappa\geq\omega_2$). In this section we will introduce variants of $\MM$ which are compatible with the failure of various instances of $F$ and $F^+$ and provably maximal in this respect. The principal applications of these axioms will come in the next two sections where we use them to separate all instances of the defined principles. We will focus on the failure of the principles at $\omega_2$ to avoid making the arguments and notation too cumbersome.
	
	\begin{mydef}
		Let $\kappa\geq\omega_2$ be regular.
		\begin{enumerate}
			\item Let $\lambda\leq\kappa$ be any cardinal and $f\colon\kappa\to\lambda$ regressive. $F_{\neg}(f,\lambda,\kappa)$ states that $f$ is not constant on any closed copy of $\omega_1$.
			\item Let $(D_i)_{i\in\omega_1}$ be a partition of $\omega_1$ and $(A_i)_{i\in\omega_1}$ a sequence of stationary subsets of $E_{\omega}^{\kappa}$. $F_{\neg}^+((A_i)_{i\in\omega_1},(D_i)_{i\in\omega_1},\kappa)$ states that there is no normal function $g\colon\omega_1\to\kappa$ such that $g[D_i]\subseteq A_i$.
		\end{enumerate}
	\end{mydef}
	
	So, roughly speaking, the previous definition holds if $f$ (or $(A_i)_{i\in\omega_1}$) witnesses that the corresponding principle fails.
	
	We can now define our versions of $\MM$:
	
	\begin{mydef}
		\begin{enumerate}
			\item Let $\lambda\leq\omega_2$ be a cardinal and $f\colon\kappa\to\lambda$ regressive. $\MM_{F(\lambda)}(f)$ states the following:
			\begin{itemize}
				\item $F(f,\lambda,\omega_2)$ holds.
				\item Whenever $\dP$ is a poset such that $\dP$ preserves stationary subsets of $\omega_1$ as well as $F_{\neg}(f,\lambda,\omega_2)$ and $\mathcal{D}$ is an $\omega_1$-sized collection of dense subsets of $\dP$ there is a filter $G\subseteq\dP$ which has nonempty intersection with any $D\in\mathcal{D}$.
			\end{itemize}
			\item Let $(D_i)_{i\in\omega_1}$ be a partition of $\omega_1$ and $(A_i)_{i\in\omega_1}$ a sequence of stationary subsets of $E_{\omega}^{\omega_2}$. $\MM_{F^+((D_i)_{i\in\omega_1})}((A_i)_{i\in\omega_1})$ states the following:
			\begin{itemize}
				\item $F^+((A_i)_{i\in\omega_1},(D_i)_{i\in\omega_1},\kappa)$ holds.
				\item Whenever $\dP$ is a poset such that $\dP$ preserves stationary subsets of $\omega_1$ as well as $F_{\neg}^+((A_i)_{i\in\omega_1},(D_i)_{i\in\omega_1},\omega_2)$ and $\mathcal{D}$ is an $\omega_1$-sized collection of dense subsets of $\dP$ there is a filter $G\subseteq\dP$ which has nonempty intersection with any $D\in\mathcal{D}$.
			\end{itemize}
		\end{enumerate}
	\end{mydef}
	
	The main result of this section is that any such principle is consistent. Thanks to Observation \ref{MainObs} showing this is surprisingly straightfoward: We obtain a model of any such principle by forcing over the standard model of Martin's Maximum with the poset adding a witness to the failure of the corresponding property.
	
	We first give some background on Martin's Maximum. The notions of \emph{proper} and \emph{semiproper forcing} were introduced by Shelah. Recall that for any poset $\dP$, any set $M$ and any $\dP$-generic filter $G$, we let $M[G]:=\{\tau^G\;|\;\tau\in M\text{ is a }\dP\text{-name}\}$.
	
	\begin{mydef}
		Let $\dP$ be a poset, $M\prec H(\Theta)$ with $\dP\in M$ and $p\in\dP$.
		\begin{enumerate}
			\item $p$ is \emph{$(M,\dP)$-generic} if $p\Vdash M[\Gamma]\cap V=M$.
			\item $p$ is \emph{$(M,\dP)$-semigeneric} if $p\Vdash M[\Gamma]\cap\omega_1=M\cap\omega_1$.
		\end{enumerate}
		$\dP$ is \emph{proper} (\emph{semiproper}) if for any sufficiently large cardinal $\Theta$ there is a club $C\subseteq[H(\Theta)]^{<\omega_1}$ such that whenever $M\in C$ and $p\in M\cap\dP$ there is an $(M,\dP)$-generic ($(M,\dP)$-semigeneric) condition $q\leq p$.
	\end{mydef}
	
	It can be shown that any proper notion of forcing preserves stationary subsets of $[X]^{<\omega_1}$ for all sets $X$ and any semiproper notion of forcing preserves stationary subsets of $\omega_1$. Moreover, properness and semiproperness is preserved by iterations with suitable support: For proper forcing, countable support suffices while semiproperness is preserved by iterations with \emph{revised countable support} (as stated in section 2, we will not go into that definition here).
	
	The other ingredient we need is the concept of a \emph{Laver function}:
	
	\begin{mydef}
		Let $\kappa$ be a supercompact cardinal. $f\colon\kappa\to V_{\kappa}$ is a \emph{Laver function} if for any set $x$ and any $\lambda\geq|\tcl(x)|$ there is an elementary embedding $j\colon V\to M$, where $M$ is transitive, such that $j(\kappa)>\lambda$, ${}^{\lambda}M\subseteq M$ and $j(f)(\kappa)=x$.
	\end{mydef}
	
	Lastly, we need that certain notions of forcing preserve witnesses to the failure of $F$ (or $F^+$).
	
	We start with the following well-known result:
	
	\begin{mylem}\label{ProperStatPres}
		Let $\gamma$ be an ordinal with $\cf(\gamma)>\omega$ and $A\subseteq E_{\omega}^{\gamma}$ stationary. Let $\dP$ be a poset.
		\begin{enumerate}
			\item If $\dP$ is proper, $A$ is stationary in $V[\dP]$.
			\item If $\dP$ is semiproper and $|\gamma|=\omega_1$ in $V$, $A$ is stationary in $V[\dP]$.
		\end{enumerate}
	\end{mylem}
	
	\begin{proof}
		Assume toward a contradiction that $\dot{C}$ is a $\dP$-name for a club such that some $p\in\dP$ forces $\dot{C}\cap\check{A}=\emptyset$. In $V$, let $D\subseteq[H(\Theta)]^{<\omega_1}$ consist of those $M\prec (H(\Theta),\in,\dot{C})$ such that $\gamma,\dP,p\in M$ and $M$ witnesses the (semi-) properness of $\dP$. $D$ is club. Ergo there is $M\in D$ such that $\sup(M\cap\gamma)\in A$. Let $q\leq p$ be $(M,\dP)$-(semi-)generic. Let $G$ be $\dP$-generic containing $q$. In $V[G]$, $M[G]$ is a limit point of $\dot{C}^G$ and thus $\sup(M[G]\cap\gamma)\in\dot{C}^G$. If $\dP$ is proper we have $M[G]\cap\gamma=M\cap\gamma$ and obtain a contradiction.
		
		If $\dP$ is semiproper and $|\gamma|=\omega_1$, $M$ contains a bijection $f$ between $\omega_1$ and $\gamma$ (and $M[G]$ thus does as well). By elementarity, $f\uhr(M\cap\omega_1)$ is a bijection between $M\cap\omega_1$ and $M\cap\gamma$ and $f\uhr(M[G]\cap\omega_1)$ is a bijection between $M[G]\cap\omega_1$ and $M[G]\cap\gamma$. Because $M[G]\cap\omega_1=M\cap\omega_1$, $M[G]\cap\gamma=M\cap\gamma$ and we proceed as above.
	\end{proof}
	
	\begin{mylem}\label{ProperNotAddingSeq}
		Let $\gamma$ be an ordinal with $\cf(\gamma)>\omega$ and $A\subseteq\gamma$ such that $A$ does not contain a closed copy of $\omega_1$.
		\begin{enumerate}
			\item If $\dP$ is proper, $A$ does not contain a closed copy of $\omega_1$ in $V[\dP]$.
			\item If $\dP$ is semiproper and $|\gamma|=\omega_1$ in $V$, $A$ does not contain a closed copy of $\omega_1$ in $V[\dP]$.
		\end{enumerate}
	\end{mylem}
	
	\begin{proof}
		We prove both parts almost simultaneously.
		
		Assume toward a contradiction that there is a $\dP$-name $\dot{f}$ such that some $p\in\dP$ forces $\dot{f}$ to be a normal function from $\omega_1$ into $\check{A}$. Assume $p$ decides $\dot{f}=\check{\beta}$ for some $\beta$ (we note that $\beta=\gamma$ is possible). We have $\cf(\beta)=\omega_1$ in $V[\dP]$ and thus $\cf(\beta)>\omega$ in $V$. However, since $A$ does not contain a closed copy of $\omega_1$ in $V$, $E_{\omega}^{\beta}\smallsetminus A$ is stationary in $\beta$ in $V$, but this stationarity is destroyed in $V[\dP]$. In any case, we obtain a contradiction using Lemma \ref{ProperStatPres}.
	\end{proof}
	
	Interestingly, we need that $\gamma$ has size $\omega_1$ in $V$ and not in $V[G]$: For any stationary $A\subseteq E_{\omega}^{\omega_2}$, there is a (at least consistently) semiproper poset that adds a closed copy of $\omega_1$ contained in $A$ while collapsing $\omega_2$ to $\omega_1$.
	
	We can now show the consistency of our versions of $\MM$. By the \emph{standard iteration to force Martin's Maximum} $\dP_{MM}$ we mean the following: Let $\kappa$ be supercompact and $f\colon\kappa\to V_{\kappa}$ a Laver function (which always exists). Let $(\dP_{\alpha},\dot{\dQ}_{\alpha})_{\alpha<\kappa}$ be a revised countable support iteration such that $\dot{\dQ}_{\alpha}=f(\alpha)$ whenever $\alpha$ is inaccessible and $f(\alpha)$ is a $\dP_{\alpha}$-name for a semiproper poset and $\dot{\dQ}_{\alpha}$ is the Levy collapse of $2^{|\dP_{\alpha}|}$ to $\omega_1$ otherwise.
	
	We first do the proof for any instance of $\MM_{F(\lambda)}(f)$.
	
	\begin{mysen}\label{MMFLambda}
		Let $\kappa$ be a supercompact cardinal and $\lambda\in\omega_1+1\cup\{\kappa\}$. There is a forcing extension where $\kappa=\omega_2$ and there is a function $f\colon\kappa\to \lambda$ such that $\MM_{F(\lambda)}(f)$ holds.
	\end{mysen}
	
	\begin{proof}
		Let $\dP:=\dP_{MM}*\dot{\dP}_{F(\lambda)}(\dot{\omega}_2)$, where $\dP_{MM}$ is the standard iteration to force Martin's Maximum (as above). We will verify that $\dP$ gives us the required model.
		
		The first step is showing that after forcing with $\dP$, any stationary set preserving forcing is semiproper. This is done using a similar proof as for Lemma 3 in \cite{ForemanMagidorShelahMM}. The idea is that in $V[\dP]$ (just like in $V[\dP_{MM}]$) any ground-model supercompactness embedding $j\colon V\to M$ can be extended to $V[\dP]\to M[j(\dP)]$ in an extension of $V[\dP]$ using a semiproper poset.
		
		So let $\dot{\dQ}$ be a $\dP$-name for a forcing notion which preserves stationary subsets of $\omega_1$ but is not semiproper. Let $G$ be any $\dP$-generic filter. By the assumption there is a stationary set $A\subseteq[H^{V[G]}(\Theta)]^{<\omega_1}$ (for $\Theta$ sufficiently large) such that for any $N\in A$ there is some $p\in N\cap\dP$ such that there is no $(N,\dP)$-semigeneric condition below $p$. By the normality of the club filter we can assume that there is one such $p$ for all $N\in A$. Assume for simplicity that $p=1_{\dQ}$ and let $\gamma:=|H^{V[G]}(\Theta)|$.
		
		In $V$, let $j\colon V\to M$ be any $\gamma^+$-supercompact embedding (i.e. $j(\kappa)>\gamma^+$ and ${}^{\gamma^+}M\subseteq M$) such that $j(l)(\kappa)$ is a $\dP_{\kappa}$-name for $\dP_{F(\lambda)}(\dot{\omega}_2)*\dot{\Coll}(\omega_1,\gamma)$. Because $\dP_{F(\lambda)}(\dot{\omega}_2)*\dot{\Coll}(\omega_1,\gamma)$ is countably closed, it is even proper and therefore $j(\dP_{MM})$ is isomorphic to $\dP*\dP_{F(\lambda)}(\dot{\omega}_2)*\dot{\Coll}(\omega_1,\gamma)*\dot{\dR}$, where $\dot{\dR}$ is forced to be a revised countable support iteration of semiproper forcings (and thus semiproper itself).
		
		Now we show that $j$ lifts to an embedding of $V[G]$. First of all, let $G_{MM}$ be the $\dP_{MM}$-generic filter induced by $G$. Let $H_{MM}$ be any $j(\dP_{MM})$-generic filter containing $G$. Then $j$ lifts to $j\colon V[G_{MM}]\to M[H_{MM}]$. In $M[H_{MM}]$, let $q$ be the union of $G_{F(\lambda)}$, where $G_{F(\lambda)}$ is the $\dP_{F(\lambda)}(\dot{\omega}_2)^{G_{MM}}$-generic filter induced by $H_{MM}$. So $G=G_{MM}*G_{F(\lambda)}$. In $V[G_{MM}*G_{F(\lambda)}]$ (and in $M[G_{MM}*G_{F(\lambda)}]$), $q$ is not constant on any closed copy of $\omega_1$.
		
		We will now show that this is still the case in $M[H_{MM}]$. First of all, let $H_{\Coll}$ be the $\dot{\Coll}(\omega_1,\gamma)^{G_{MM}*G_{F(\lambda)}}$-generic filter induced by $H_{MM}$. Because $\dot{\Coll}(\omega_1,\gamma)^{G_{MM}*G_{F(\lambda)}}$ is countably closed and hence proper, $q$ is still not constant on any closed copy of $\omega_1$ in $M[G_{MM}*G_{F(\lambda)}*H_{\Coll}]$. Moreover, after forcing with $\Coll(\omega_1,\gamma)^{G_{MM}*G_{F(\lambda)}}$, $q$ is a function on an ordinal in $\omega_2$ with cofinality $\omega_1$. Thus, letting $H_{\dR}$ be the $\dot{\dR}^{G_{MM}*G_{F(\lambda)}*H_{\Coll}}$-generic filter induced by $H$, $q$ is still not constant on a closed copy of $\omega_1$ in $M[G_{MM}*G_{F(\lambda)}*H_{\Coll}*H_{\dR}]=M[H_{MM}]$ because $\dot{\dR}^{G_{MM}*G_{F(\lambda)}*H_{\Coll}}$ is semiproper.
		
		In summary, $q$ is a condition in $\dP_{F(\lambda)}(\dot{\omega}_2)^{H_{MM}}$. Ergo, letting $H_{F(\lambda)}$ be any $\dP_{F(\lambda)}(\dot{\omega}_2)^{H_{MM}}$-generic filter containing $q$, $j[G_{MM}*G_{F(\lambda)}]\subseteq H_{MM}*H_{F(\lambda)}$ and so $j$ lifts to $j\colon V[G_{MM}*G_{F(\lambda)}]\to M[H_{MM}*H_{F(\lambda)}]$. Let $H:=H_{MM}*H_{F(\lambda)}$ (recall $G=G_{MM}*G_{F(\lambda)}$), so $j\colon V[G]\to M[H]$.
		
		In $M[H]$, $A$ is still stationary in $[H^{V[G]}(\Theta)]^{<\omega_1}$: $A$ is stationary in $M[G]$. $M[H_{MM}]$ has been obtained by forcing with $(\dot{\Coll}(\omega_1,\gamma)*\dot{\dR})^G$ over $M[G]$. $\dot{\Coll}(\omega_1,\gamma)^G$ is proper, so it preserves the stationarity of $A$ and makes $A$ (morally) a stationary subset of $\omega_1$, so its stationarity is preserved by the semiproper forcing $\dot{\dR}^{G*H_{\Coll}}$. Lastly, $A$ is stationary in $M[H]$ by the countable closure of $\dP_{F(\lambda)}(\dot{\omega}_2)^{H_{MM}}$.
		
		By elementarity, $j(\dot{\dQ})^H$ preserves stationary subsets of $\omega_1$ in $M[H]$. Thus, if $I$ is $j(\dot{\dQ})^H$-generic, $A$ is stationary in $[H(\Theta)^{V[G]}]^{<\omega_1}$ in $M[H*I]$. In $M[H*I]$, there is a club set $C\subseteq H(j(\Theta))^{M[H*I]}$ such that whenever $N\in C$ and $\tau\in N\cap H(\Theta)^{V[G]}$ is a $\dot{\dQ}^G$-name for an element of $\omega_1$, then $j(\tau)^{H*I}\in N$. Ergo there is $N\in C$ such that $N\cap H(\Theta)^{V[G]}\in A$. In $M[H]$, let $q\in j(\dQ)^H$ force that such an $N$ with $N':=N\cap H(\Theta)^{V[G]}\in A$ exists (for some fixed $N'$). Then $q$ forces that whenever $\tau\in N'$ is a $\dot{\dQ}^G$-name for an element of $\omega_1$, then $j(\tau)^{H*I}\in N'$. Since $N'$ is countable, $j(N')=j[N']$, Ergo $q$ forces that whenever $\tau\in j(N')$ is a $j(\dot{\dQ})^{H}$-name for an element of $\omega_1$, then $\tau^{H*I}\in j(N')$. So in $M[H]$ there exists a semigeneric condition for $j(N')$ and thus in $V[G]$ there exists a semigeneric condition for $N'$. This contradicts our assumption.
		
		A very similar argument shows that $\MM_{F(\lambda)}(f)$ holds in $V[G]$: Let $\dot{\dQ}$ be a $\dP$-name for a forcing that preserves stationary subsets of $\omega_1$ and forces $F(f,\lambda,\omega_2)$ and $\dot{\mathcal{D}}$ a $\dP$-name for an $\omega_1$-sized collection of dense subsets of $\dot{\dQ}$. Let $j\colon V\to M$ be a $\gamma$-supercompact embedding (where $\gamma\geq|\dot{\dQ}|$) such that $j(l)(\kappa)$ is equal to $\dP_{F(\lambda)}(\omega_2)*\dot{\dQ}*\dot{\Coll}(\omega_1,\dot{\omega}_2)$. Just as in the last proof, there is a $j(\dP)$-generic filter $H$ such that $j$ lifts to $j\colon V[G]\to M[H]$: $\dot{\dQ}$ preserves that the generic function added by $\dP_{F(\lambda)}(\omega_2)$ is not constant on a closed copy of $\omega_1$ by assumption, this is further preserved by $\Coll(\omega_1,\omega_2)$ by the properness and lastly by the tail of the iteration as before. In $M[H]$ there is a filter $L$ which is generic for $\dot{\dQ}^G$ over $M[G]$. Clearly this filter intersects every element of $\dot{\mathcal{D}}^G$. Moreover, $j(\dot{\mathcal{D}}^G)=j[\dot{\mathcal{D}}^G]$ because it has size $\omega_1$. Ergo, $j[L]$ (which is in $M[H]$ by the closure) intersects every element of $j(\dot{\mathcal{D}}^G)$. By elementarity there is a filter in $V[G]$ for the poset $\dot{\dQ}^G$ which intersects every element of $\dot{\mathcal{D}}^G$. Thus $\MM_{F(\lambda)}(f)$ holds.
	\end{proof}
	
	We now turn to $F^+$. We first have the following analogue of Lemma \ref{ProperNotAddingSeq}:
	
	\begin{mylem}\label{ProperNotAddingSeq2}
		Let $\gamma$ be an ordinal with $\cf(\gamma)>\omega$. Let $(D_i)_{i\in\omega_1}$ be a partition of $\omega_1$, $(A_i)_{i\in\omega_1}$ a sequence of stationary subsets of $E_{\omega}^{\gamma}$ and $\dP$ a poset. Assume there is no normal function $g\colon\omega_1\to\gamma$ such that $g[D_i]\subseteq A_i$ for all $i\in\omega_1$.
		\begin{enumerate}
			\item If $\dP$ is proper, this is preserved by $\dP$.
			\item If $\dP$ is semiproper and $|\gamma|=\omega_1$, this is preserved by $\dP$.
		\end{enumerate}
	\end{mylem}
	
	\begin{proof}
		Assume toward a contradiction that there is $p\in\dP$ which forces some $\dP$-name $\dot{g}$ to be a normal function from $\omega_1$ to $\gamma$ such that $\dot{g}[D_i]\subseteq A_i$ for all $i\in\omega_1$. Assume also that $p$ forces $\sup(\im(\dot{g}))=\check{\delta}$ for some $\delta\leq\gamma$ (as before, $\delta=\gamma$ is possible).
		
		Let $(M_i)_{i\in\omega_1}$ be an $\in$-increasing and continuous sequence of elements of $[H(\Theta)]^{<\omega_1}$ such that $M_i\prec(H(\Theta),\in)$, $\dot{g},p,\delta,\gamma,\overline{A},\overline{D}\in M_i$ and $M_i$ witnesses (semi-)properness for any $i\in\omega_1$.
		
		\begin{myclaim}
			There is $i,j\in\omega_1$ such that $M_i\cap\omega_1\in D_j$ and $\sup(M_i\cap\delta)\notin A_j$.
		\end{myclaim}
		
		\begin{proof}
			Otherwise, we could define a normal function $f\colon\omega_1\to\gamma$ by letting $f(M_i\cap\omega_1):=\sup(M_i\cap\delta)$ for $i\in\omega_1$ and ``filling in the gaps'' as in the proof of Lemma \ref{SRPImpliesFPlus}.
		\end{proof}
		
		So let $i\in\omega_1$ be in the claim. Let $q\leq p$ be $(M,\dP)$-(semi-)generic and let $G$ be a $\dP$-generic filter containing $q$. We now work in $V[G]$. Since $\dot{g}^G\in M_i[G]$ and $M_i[G]\prec(H^{V[G]}(\Theta),\in)$, the image of $M_i[G]\cap\omega_1$ under $\dot{g}^G$ is unbounded in $M_i[G]\cap\delta$, so by normality $\dot{g}^G(M_i[G]\cap\omega_1)=\sup(M_i[G]\cap\delta)$. In any case, we obtain that $M_i[G]\cap\omega_1=M_i\cap\omega_1$ and $M_i[G]\cap\delta=M_i\cap\delta$ and thus a contradiction.
	\end{proof}
	
	Using Lemma \ref{ProperNotAddingSeq2} in place of Lemma \ref{ProperNotAddingSeq} in the proof of Theorem \ref{MMFLambda}, one can see that the following holds:
	
	\begin{mysen}
		Let $\kappa$ be a supercompact cardinal and $\dP_{MM}$ the standard iteration to forced Martin's Maximum defined from $\kappa$. In $V[\dP_{MM}]$, let $(D_i)_{i\in\omega_1}$ be any partition of $\omega_1$. Then there is a forcing extension of $V[\dP_{MM}]$ where there is a sequence $(A_i)_{i\in\omega_1}$ of stationary subsets of $E_{\omega}^{\omega_2}$ such that $\MM_{F^+((D_i)_{i\in\omega_1})}((A_i)_{i\in\omega_1})$ holds.
	\end{mysen}
	
	\begin{proof}
		Simply force with $\dP_{MM}*\dP_{F^+((\dot{D}_i)_{i\in\omega_1})}(\dot{\omega}_2)$.
	\end{proof}
	
	\section{Separating Instances of $F$}
	
	We now use $\MM_{F(\lambda)}(f)$ in order to separate all possible instances of $F(\cdot,\omega_2)$.
	
	\begin{mysen}\label{FDistinction}
		Assume $\lambda\leq\omega_2$ is a cardinal and $f\colon\kappa\to\lambda$ is a regressive function such that $\MM_{F(\lambda)}(f)$ holds and $f^{-1}[\{i\}]\cap E_{\omega}^{\omega_2}$ is stationary in $\omega_2$ for any $i\in\lambda$. Then $F(\lambda',\kappa)$ holds for all cardinals $\lambda'<\lambda$.
	\end{mysen}
	
	\begin{proof}
		Let $\lambda'<\lambda$ be a cardinal and $g\colon\kappa\to\lambda'$ a regressive function.
		\begin{myclaim}
			There are $i<\lambda'$ and $j,j'<\lambda$ such that both $g^{-1}[\{i\}]\cap f^{-1}[\{j\}]\cap E_{\omega}^{\omega_2}$ and $g^{-1}[\{i\}]\cap f^{-1}[\{j'\}]\cap E_{\omega}^{\omega_2}$ are stationary in $\omega_2$.
		\end{myclaim}
		
		\begin{proof}
			Let $j<\lambda$. Because $\omega_2=\bigcup_{i<\lambda'}g^{-1}[\{i\}]$ and $\lambda'<\lambda\leq\omega_2$, there is $h(j)<\lambda'$ such that $g^{-1}[\{h(j)\}]\cap f^{-1}[\{j\}]\cap E_{\omega}^{\omega_2}$ is stationary. By the pigeonhole principle there are $j,j'<\lambda$ with $h(j)=h(j')$. We can take $i:=h(j)=h(j')$.
		\end{proof}
		
		Let $\dP:=\dP(g^{-1}[\{i\}]\cap E_{\omega}^{\omega_2})$, i.e. the forcing shooting a closed copy of $\omega_1$ into $g^{-1}[\{i\}]\cap E_{\omega}^{\omega_2}$.
		
		\begin{myclaim}
			$\dP$ forces $F(f,\lambda,\omega_2)$.
		\end{myclaim}
		
		\begin{proof}
			Let $\dot{h}$ be a $\dP$-name for a normal function from $\omega_1$ into $\omega_2$ and $p\in\dP$. We will show that there is a condition $q\leq p$ that forces $\check{f}(\dot{h}(\check{\alpha}))=\check{j}$ for some $\alpha\in\omega_1$. A symmetrical argument shows that there is a further $r\leq q$ that forces $\check{f}(\dot{h}(\check{\beta}))=\check{j}'$, so that $\check{f}$ is forced by $r$ not to be constant on $\im(\dot{h})$.
			
			Since $g^{-1}[\{i\}]\cap E_{\omega}^{\omega_2}\cap f^{-1}[\{j\}]$ is stationary in $\omega_2$, we can find $M\prec H(\Theta)$ countable containing $\dP,p$ and $\dot{h}$ such that $\sup(M\cap\omega_2)\in g^{-1}[\{i\}]\cap E_{\omega}^{\omega_2}\cap f^{-1}[\{j\}]$ (as in the proof of Lemma \ref{PAStatPres}). Let $(D_n)_{n\in\omega}$ enumerate all open dense subsets of $\dP$ lying in $M$. Find a descending sequence $(p_n)_{n\in\omega}$ with $p_0\leq p$ such that $p_n\in D_n$ for all $n\in\omega$. Define $\delta_1:=M\cap\omega_1$ and $\delta_2:=\sup(M\cap\omega_2)$ and let $q:=\bigcup_np_n\cup\{(\delta_1,\delta_2)\}$.
			
			\begin{mysclaim}
				$q$ is a condition in $\dP$.
			\end{mysclaim}
			
			\begin{proof}
				Let $q':=\bigcup_np_n$. It suffices to show that $\dom(q')=M\cap\omega_1$ and $\sup(\im(q'))=\sup(M\cap\omega_2)$. For any $n\in\omega$, $p_n\in M$, so $p_n\subseteq M$ as it is countable, which implies $q'\subseteq M$ and thus $\dom(q')\subseteq M\cap\omega_1$ and $\im(q')\subseteq M\cap\omega_2$. On the other hand, for any $\alpha<M\cap\omega_1$, the set of all conditions $r$ with $\alpha\in\dom(r)$ is open dense in $\dP$ and in $M$, so $\alpha\in\dom(q')$. Additionally, for any $\alpha<\sup(M\cap\omega_2)$, there is $\beta>\alpha$ with $\beta\in M\cap\omega_2$. The set of all conditions $r$ with $\beta<\sup(\im(r))$ is open dense in $\dP$ and in $M$, so $\beta<\sup(\im(q'))$.
			\end{proof}
			
			A similar argument shows the following:
			
			\begin{mysclaim}
				$q$ forces $\dot{h}(\check{\delta}_1)=\check{\delta}_2$.
			\end{mysclaim}
			
			\begin{proof}
				Because $\dot{h}$ is forced to be continuous, it suffices to show that $q$ forces $\sup(\dot{h}[\check{\delta}_1])=\check{\delta}_2$. Whenever $\alpha\in\delta_1$, the set $D_{\alpha}$ of conditions $r$ with decide $\dot{h}(\check{\alpha})$ is open dense in $\dP$ and in $M$, so there is $n\in\omega$ with $p_n\in D_{\alpha}$. The corresponding value that is decided for $\dot{h}(\check{\alpha})$ is in $M$ as well by elementarity. This shows $\sup(\dot{h}[\check{\delta}_1])\leq\check{\delta}_2$.  On the other hand, given any $\beta\in M\cap\omega_2$ the set $D_{\beta}$ of conditions $r$ which force $\sup(\im(\dot{h}))\geq\check{\beta}$ is open dense in $\dP$ and in $M$, so again there is $n\in\omega$ with $p_n\in D_{\beta}$. This shows the equality.
				\end{proof}
				
				So in particular $q$ forces that $\check{\delta}_2\in\im(\dot{h})$. Because $\delta_2\in f^{-1}[\{j\}]$, $q$ forces that $\check{f}(\dot{h}(\check{\alpha}))=\check{j}$ for some $\alpha\in\omega_1$.
			\end{proof}
		
		Now let, for $\alpha<\omega_1$, $D_{\alpha}:=\{r\in\dP\;|\;\alpha\in\dom(r)\}$. It follows that any $D_{\alpha}$ is open dense in $\dP$ (using similar arguments to before). Thus, by $\MM_{F(\lambda)}(f)$ there is a filter $G\subseteq\dP$ which intersects any $D_{\alpha}$. It follows that $\bigcup G$ is a normal function from $\omega_1$ into $\omega_2$ such that $g$ is constant (with value $i$) on $\im(\bigcup G)$.
	\end{proof}
	
	As a corollary, we obtain that there is no implication from $F(\lambda',\omega_2)$ to $F(\lambda,\omega_2)$ when $\lambda'<\lambda$, i.e. it is consistent that there is a partition of $\omega_2$ into $17$ parts such that none of them contain a closed copy of $\omega_1$ while there is no such partition of $\omega_2$ into any fewer parts. Let us note the following two results of a similar type: In \cite{CumForeMagCanonicalStructureTwo}, Theorem 8.1, Cummings, Foreman and Magidor show that for any $\eta\leq\omega_1$ it is consistent that there are $\eta$ stationary subsets of $\omega_2\cap\cof(\omega)$ which do not reflect simultaneously while every collection of fewer than $\eta$ stationary subsets of $\omega_2\cap\cof(\omega)$ does reflect simultaneously. In \cite{CumForeMagSquareScalesStatRefl}, the same authors show that there is no implication from $\square_{\omega_1,\lambda}$ to $\square_{\omega_1,\lambda'}$ for $\lambda'<\lambda$.
	
	\begin{mycol}
		Assume $V$ is the standard model for Martin's Maximum and $\lambda\leq\omega_2$. There is a forcing extension where no cardinals are collapsed, $F(\lambda,\kappa)$ fails and $F(\lambda',\kappa)$ holds for all $\lambda'<\lambda$.
	\end{mycol}
	
	\begin{proof}
		The forcing extension is obtained straightforwardly using the poset $\dP_{F(\lambda)}(\omega_2)$. A simple genericity argument shows that the generic function obtains every value stationarily often, so we can apply Theorem \ref{FDistinction}.
	\end{proof}
	
	\section{Separating Instances of $F^+$ and $\SRP$}
	
	As Fuchs showed in \cite{FuchsCanonicalFragments}, there is a natural implication between instances of $\SRP$ and instances of $F^+$: As we defined in the introduction, given a partition $(D_i)_{i\in\omega_1}$ of $\omega_1$ and a sequence $(A_i)_{i\in\omega_1}$ of stationary subsets of $E_{\omega}^{\kappa}$, there is for any $\Theta\geq\kappa$ a canonical subset of $[H(\Theta)]^{<\omega_1}$ associated with the pair $((D_i)_{i\in\omega_1},(A_i)_{i\in\omega_1})$:
	\begin{multline*}
		S((D_i)_{i\in\omega_1},(A_i)_{i\in\omega_1},\Theta):=\hspace*{\fill} \\
		\hspace*{\fill}\{M\in[H(\Theta)]^{<\omega_1}\;|\;\forall i\in\omega_1(M\cap\omega_1\in D_i\to\sup(M\cap\kappa)\in A_i)\}
	\end{multline*}
	
	We also define $\mathcal{S}((D_i)_{i\in\omega_1},\kappa,\Theta)$ as the collection of $S((D_i)_{i\in\omega_1},(A_i)_{i\in\omega_1},\Theta)\cap C$ for all sequences $(A_i)_{i\in\omega_1}$ of stationary subsets of $E_{\omega}^{\kappa}$ and clubs $C\subseteq[H(\Theta)]^{<\omega_1}$.
	
	We first answer a question from Fuchs' paper by showing that this set is projective stationary for any partition of $\omega_1$:
	
	\begin{mysen}\label{ProjStat}
		Let $(D_i)_{i\in\omega_1}$ be a partition of $\omega_1$, $\kappa\geq\omega_2$ regular and $(A_i)_{i\in\omega_1}$ a sequence of stationary subsets of $E_{\omega}^{\kappa}$. Then $S((D_i)_{i\in\omega_1},(A_i)_{i\in\omega_1},\Theta)$ is projective stationary.
	\end{mysen}
	
	\begin{proof}
		Let $C\subseteq[H(\Theta)]^{<\omega_1}$ be club and $S\subseteq\omega_1$ stationary. Our aim is to find $M\in C\cap S((D_i)_{i\in\omega_1},(A_i)_{i\in\omega_1},\Theta)$ with $M\cap\omega_1\in S$. Let $\mathcal{H}$ be an expansion of $H(\Theta)$ coding all necessary information such that any $M\prec \mathcal{H}$ is in $C$.
		
		Let $T$ be the tree $\kappa^{<\omega}$, i.e. $T$ consists of all finite sequences of elements of $\omega_2$. Define a function $f\colon[T]\to\omega_1$ by letting $f(b)$ be the minimum $\alpha\in S$ such that $\Hull^{\mathcal{H}}(b\cup\alpha)\cap\omega_1=\alpha$. Such an $\alpha$ exists as the set of all $\beta$ with $\Hull^{\mathcal{H}}(b\cup\beta)\cap\omega_1=\beta$ is club in $\omega_1$.
		
		It follows that for any $\alpha<\omega_1$, the set $f^{-1}[\alpha+1]$ is closed in the tree topology: If $f(b)>\alpha$ then there is a finite $\eta\sqsubset b$ such that $\Hull^{\mathcal{H}}(\eta\cup\alpha)\cap\omega_1>\alpha$. Consequently, $\{c\in\kappa^{<\omega}\;|\;\eta\sqsubset c\}$ is open, contains $b$ and is disjoint from $f^{-1}[\alpha+1]$.
		
		By Lemma 3.5 in \cite[chapter XI]{ShelahProperImproper}, there exists a tree $T''\leq T$ and $\beta\in\omega_1$ such that $|\Suc_{T''}(\eta)|=\kappa$ for any $\eta\in T''$ and $[T'']\subseteq f^{-1}[\beta+1]$. Let $\alpha\in\omega_1$ be minimal such that for some $\eta\in T''$ and any branch $b$ of $T''$ containing $\eta$, $f(b)\leq\alpha$. Let
		$$T':=T''\uhr\eta:=\{\eta'\in T''\;|\;\eta'\sqsubseteq\eta\vee \eta\sqsubseteq\eta'\}$$
		It follows from the minimality of $\alpha$ that for any $\eta\in T'$ there is a branch $b$ of $T'$ with $f(b)=\alpha$: We know $f(b)\leq\alpha$ and if $f(b)<\alpha$ for all such branches, $\eta$ would witness that $\alpha$ was not minimal.
		
		However, this actually implies that $\Hull^{\mathcal{H}}(b\cup\alpha)\cap\omega_1=\alpha$ for any branch $b$ of $T'$: Otherwise there would be a finite $\eta\sqsubset b$ such that $\Hull^{\mathcal{H}}(\eta\cup\alpha)\cap\omega_1>\alpha$ but then we could never extend $\eta$ to a branch $b'$ with $f(b')=\alpha$.
		
		Now consider the following clubs:
		
		\begin{enumerate}
			\item Let $C_0\subseteq\kappa$ consist of all those $\gamma\in\kappa$ such that $\sup(\Hull^{\mathcal{H}}(\alpha\cup s)\cap\kappa)<\gamma$ for all $s\in[\gamma]^{<\omega}$.
			\item Let $C_1\subseteq\kappa$ consist of all those $\gamma\in\omega_2$ such that whenever $\eta\in T'$ is above the stem with $\eta\in\gamma^{<\omega}$ and $\alpha\in\gamma$ there is $\beta\in(\alpha,\gamma)$ such that $\eta^{\frown}\beta\in T'$.
		\end{enumerate}
		
		$C_0$ is clearly club in $\kappa$. $C_1$ is club in $\kappa$ as well because whenever $M$ is an elementary submodel of $(H(\Theta),T',\in)$ has size $\omega_1$, $M\cap\kappa$ is in $C_1$ (this shows unboundedness, closure is clear).
		
		Let $i$ be such that $\alpha\in D_i$. Let $\gamma\in A_i\cap C_0\cap C_1$. Fix sequences $(\alpha_n)_{n\in\omega}$ converging to $\alpha$ and $(\gamma_n)_{n\in\omega}$ converging to $\gamma$ (this is possible as $A_{\alpha}\subseteq\cof(\omega)$).
		
		We construct a branch $b$ of $T'$ as follows: Let $\eta_0$ be the stem of $T'$. Given $\eta_n$, find $\eta_{n+1}$ in $T'$ such that $\max(\eta_{n+1})>\gamma_n$. Let $b:=\bigcup_n\eta_n$ and let $M:=\Hull^{\mathcal{H}}(\alpha\cup b)\in C$. By assumption we have $M\cap\omega_1=\alpha\in S\cap D_i$. Furthermore we have $\gamma\leq\sup(M\cap\kappa)$ by the construction of $b$ and $\gamma\geq\sup(M\cap\kappa)$ by $\gamma\in C_0$, so $\sup(M\cap\kappa)\in A_i$.
		
		Consequently, $M\in C\cap S((D_i)_{i\in\omega_1},(A_i)_{i\in\omega_1},\Theta)$ and $M\cap\omega_1\in S$.
	\end{proof}
	
	It is clear that any intersection of a projective stationary set with a club is projective stationary. Following an argument of Fuchs, this shows that $\SRP$ implies $F^+((D_i)_{i\in\omega_1},\kappa)$ for any partition $(D_i)_{i\in\omega_1}$ of $\omega_1$:
	
	\begin{mysen}\label{ThmSRP}
		Let $(D_i)_{i\in\omega_1}$ be a partition of $\omega_1$ and $\kappa\geq\omega_2$ regular. Let $\Theta\geq\kappa$ be sufficiently large. Then $\SRP(\mathcal{S}((D_i)_{i\in\omega_1},\kappa,\Theta))$ implies $F^+((D_i)_{i\in\omega_1},\kappa)$.
	\end{mysen}
	
	\begin{proof}
		Let $(A_i)_{i\in\omega_1}$ be a sequence of stationary subsets of $E_{\omega}^{\kappa}$. Let $\Theta\geq\kappa$ be large enough. Let $C\subseteq[H(\Theta)]^{<\omega_1}$ be the club of all $M\prec H(\Theta)$ such that $(D_i)_{i\in\omega_1},(A_i)_{i\in\omega_1}\in M$. Since $\SRP(\mathcal{S}((D_i)_{i\in\omega_1},\kappa,\Theta))$ holds, there is a sequence $(X_i)_{i\in\omega_1}$ of elements of $S((D_i)_{i\in\omega_1},(A_i)_{i\in\omega_1},\Theta)\cap C$ such that $X_i\in X_{i+1}$ for any $i\in\omega_1$ and $X_i=\bigcup_{j<i}X_j$ for any limit $j\in\omega_1$. Let $D\subseteq\omega_1$ be the club consisting of $X_i\cap\omega_1$ for $i\in\omega_1$ and define $f$ on $D$ by letting $f(X_i\cap\omega_1):=\sup(X_i\cap\kappa)$. Then $f$ is a normal function from $D$ into $\omega_2$ and satisfies $f[D_i\cap D]\subseteq A_i$. So all that is left is to ``fill in the gaps''.
		
		For $i\in\omega_1$, let $\dP_i$ be the forcing notion of all functions $p\colon[X_i\cap\omega_1,\alpha]$ where $\alpha<\omega_1$ such that $p(\beta)\in A_j$ whenever $\beta\in D_j\cap\dom(p)$ for all $j\in\omega_1$. Because $X_i\in X_{i+1}$, $\dP_i\in X_{i+1}$. Let $(D_n^i)_{n\in\omega}$ enumerate all open dense subsets of $\dP_i$ lying in $X_{i+1}$. Let $(p_n^i)_{n\in\omega}$ be a descending sequence of elements of $\dP_i\cap X_{i+1}$ such that $p_n^i\in D_n^i$ for any $n\in\omega$. Let $p_i:=\bigcup_np_n^i$. It follows that $\dom(p_i)=[X_i\cap\omega_1,X_{i+1}\cap\omega_1)$ and $\sup(\im(p))=\sup(X_{i+1}\cap\kappa)$ by genericity. Ergo the function $g$ defined by
		$$g:=f\cup\bigcup_{i<\omega_1}p_i$$
		is a normal function from $\omega_1$ into $\kappa$ such that $g[D_i]\subseteq A_i$ for any $i\in\omega_1$.
	\end{proof}
	
	We now investigate implications between instances of $\SRP$ and $F^+$. We first introduce two partial orders on the set of partitions of $\omega_1$:
	
	\begin{mydef}
		Let $\overline{D}=(D_i)_{i\in\omega_1}$ and $\overline{E}=(E_i)_{i\in\omega_1}$ be partitions of $\omega_1$.
		
		\begin{enumerate}
			\item We write $\overline{D}\leq\overline{E}$ if there is a function $h\colon\omega_1\to\omega_1$ such that for any $i\in\omega_1$, $D_i\subseteq E_{h(i)}$.
			\item We write $\overline{D}\leq^*\overline{E}$ if there is a function $h\colon\omega_1\to\omega_1$ and a club $C\subseteq\omega_1$ such that for any $i\in\omega_1$, $D_i\cap C\subseteq E_{h(i)}$.
		\end{enumerate}
	\end{mydef}
	
	So $\overline{D}\leq\overline{E}$ states that $\overline{D}$ refines $\overline{E}$ while $\overline{D}\leq^*\overline{E}$ states that $\overline{D}$ refines $\overline{E}$ on a club.
	
	\begin{mybem}\label{RemEqLeqStar}
		The following reformulation of $\leq^*$ is useful: $\overline{D}\leq^*\overline{E}$ if and only if there is a function $h\colon\omega_1\to\omega_1$ such that the set
		$$S(h):=\{\alpha\in\omega_1\;|\;\exists i(\alpha\in D_i\smallsetminus E_{h(i)})\}$$
		is nonstationary: Assume $\overline{D}\leq^*\overline{E}$, witnessed by $h$ and $C$. Then for any $\alpha\in C$ and any $i\in\omega_1$, if $\alpha\in C\cap D_i$, $\alpha\in E_{h(i)}$. Thus $C$ is disjoint from $S(h)$. On the other hand, assume there is $h$ such that $S(h)$ is nonstationary, with $C\subseteq\omega_1$ club and disjoint from $S(h)$. Then for any $\alpha\in C$ and all $i\in\omega_1$, $\alpha\in D_i$ implies $\alpha\in E_{h(i)}$, so $D_i\cap C\subseteq E_{h(i)}\cap C$.
	\end{mybem}
	
	This enables us to show the following:
	
	\begin{mylem}
		Assume $(D_i)_{i\in\omega_1}$ and $(E_i)_{i\in\omega_1}$ are partitions of $\omega_1$ such that $\overline{D}\leq\overline{E}$ and $\kappa\geq\omega_2$ is regular. Then $F^+(\overline{D},\kappa)$ implies $F^+(\overline{E},\kappa)$.
	\end{mylem}
	
	\begin{proof}
		Let $h$ witness $\overline{D}\leq\overline{E}$.
		
		Let $(B_i)_{i\in\omega_1}$ be a sequence of stationary subsets of $E_{\omega}^{\kappa}$. Define $(A_i)_{i\in\omega_1}$ by letting $A_i:=B_{h(i)}$. Applying $F^+(\overline{D},\kappa)$, let $f\colon\omega_1\to\kappa$ be a normal function such that $f[D_i]\subseteq A_i$. Let $j\in\omega_1$ and $\alpha\in E_j$. Let $i\in\omega_1$ be such that $\alpha\in D_i$. Then $h(i)=j$, since $D_i\subseteq E_{h(i)}$ and $(D_i)_{i\in\omega_1}$, $(E_i)_{i\in\omega_1}$ are partitions. Ergo
		$$f(\alpha)\in A_i=B_{h(i)}=B_j$$
		In summary, $f[E_j]\subseteq B_j$ for any $j\in\omega_1$.
	\end{proof}
	
	In case we have obtained $F^+$ through $\SRP$, we can get away with the weaker ordering $\leq^*$. If we had tried to do the proof above using $\leq^*$, we could only have defined our function on the club $C$ witnessing $\overline{D}\leq^*\overline{E}$. It is not clear if we can always find a function $f$ defined on $C$ that can then be extended to the whole of $\omega_1$ (it is easy to see that there exist functions which cannot be extended, e.g. if $f(\alpha+2)=f(\alpha)+2$, $f(\alpha+1)=f(\alpha)+1$ but it is possible that this is not in the required stationary set).
	
	\begin{mylem}\label{SRPImpliesFPlus}
		Assume $(D_i)_{i\in\omega_1}$ and $(E_i)_{i\in\omega_1}$ are partitions of $\omega_1$ such that $\overline{D}\leq^*\overline{E}$ and $\kappa\geq\omega_2$ is regular. Let $\Theta\geq\kappa$ be sufficiently large. Then $\SRP(\mathcal{S}((D_i)_{i\in\omega_1},\kappa,\Theta))$ implies $F^+(\overline{E},\kappa)$.
	\end{mylem}
	
	\begin{proof}
		Let $h$ and $C$ witness $\overline{D}\leq\overline{E}$.
		
		Let $(B_i)_{i\in\omega_1}$ be a sequence of stationary subsets of $E_{\omega}^{\kappa}$. Define $(A_i)_{i\in\omega_1}$ by letting $A_i:=B_{h(i)}$. Let $D\subseteq[H(\Theta)]^{<\omega_1}$ be the club of all $M\prec(H(\Theta),\in,\overline{B},\overline{A},\overline{D},\overline{E},C)$. Applying $\SRP(\mathcal{S}((D_i)_{i\in\omega_1},\kappa,\Theta))$ to $S:=S((D_i)_{i\in\omega_1},(A_i)_{i\in\omega_1})\cap D$, obtain an $\in$-increasing continuous sequence $(X_i)_{i\in\omega_1}$ of elements of $S$. Define $f$ on the club $F$ consisting of $X_i\cap\omega_1$ for any $i\in\omega_1$ by letting $f(X_i\cap\omega_1)=\sup(X_i\cap\kappa)$. As before we have $f[D_i\cap F]\subseteq A_i$ for any $i\in\omega_1$.
		
		\begin{myclaim}
			For any $i\in\omega_1$, $f[E_i\cap F]\subseteq B_i$.
		\end{myclaim}
		
		\begin{proof}
			Let $\alpha\in E_i\cap F$, i.e. $\alpha=X_j\cap\omega_1$ for some $j\in\omega_1$. As $X_j\prec(H(\Theta),\in,C)$, $X_j\cap\omega_1\in C$. Let $k\in\omega_1$ be such that $\alpha\in D_k$. Because $X_j\cap\omega_1\in C\cap D_k\subseteq E_{h(k)}$, we have $h(k)=i$. Ergo
			$$f(\alpha)\in A_k=B_{h(k)}=B_i$$
		\end{proof}
		
		Now we can proceed as in the proof of Theorem \ref{ThmSRP} to show that we can extend $f$ to all of $\omega_1$.
	\end{proof}
	
	It turns out that this is all that is provable (and thus that $\leq^*$ is the exact relation that characterizes the implication between $\SRP$ and $F^+$).
	
	\begin{mysen}\label{DistFPlusSRP}
		Assume $(D_i)_{i\in\omega_1}$ is a partition of $\omega_1$ and $(A_i)_{i\omega_1}$ is a disjoint sequence of stationary subsets of $E_{\omega}^{\omega_2}$ such that $\MM_{F^+((D_i)_{i\in\omega_1})}((A_i)_{i\in\omega_1})$ holds. Let $(E_i)_{i\in\omega_1}$ be a partition of $\omega_1$. Then the following are equivalent:
		
		\begin{enumerate}
			\item $\overline{E}\not\leq^*\overline{D}$
			\item $\SRP(\mathcal{S}((E_i)_{i\in\omega_1},\omega_2,\Theta))$ holds for all $\Theta\geq\kappa$.
		\end{enumerate}
	\end{mysen}
	
	\begin{proof}
		That (2) implies (1) follows from Lemma \ref{SRPImpliesFPlus} together with the fact that $\MM_{F^+((D_i)_{i\in\omega_1})}((A_i)_{i\in\omega_1})$ implies $\neg F^+((D_i)_{i\in\omega_1})$.
		
		Let $(B_i)_{i\in\omega_1}$ be a sequence of stationary subsets of $E_{\omega}^{\omega_2}$ and $C\subseteq[H(\Theta)]^{<\omega_1}$ any club. Let $\dP$ be the forcing shooting an $\in$-increasing and continuous sequence through $S:=S((E_i)_{i\in\omega_1},(B_i)_{i\in\omega_1},\Theta)\cap C$, i.e. $\dP$ consists of continuous $\in$-increasing functions $p\colon\alpha+1\to S$, where $\alpha<\omega_1$. It is not hard to see (and standard) that, since $S$ is projective stationary by Theorem \ref{ProjStat}, $\dP$ preserves stationary subsets of $\omega_1$. We now aim to show that $\dP$ does not add a normal function $g\colon\omega_1\to\omega_2$ with $g[D_i]\subseteq A_i$ for all $i\in\omega_1$. Assume otherwise and let $\dot{g}$ be a name for such a function.
		
		\begin{myclaim}
			For any $i\in\omega_1$ there exists at most one $h(i)\in\omega_1$ such that $B_i\smallsetminus A_{h(i)}$ is nonstationary.
		\end{myclaim}
		
		\begin{proof}
			Assume otherwise that $B_i\smallsetminus A_{j}$ and $B_i\smallsetminus A_{j'}$ are nonstationary for $j\neq j'$. Then $B_i=B_i\smallsetminus (A_j\cap A_{j'})=(B_i\smallsetminus A_j)\cup (B_i\smallsetminus A_{j'})$ is nonstationary as well, a contradiction.
		\end{proof}
		
		Consider the function $h$ defined by using the claim (and letting $h(i):=0$ if such a value does not exist). Then for all $i,j\in\omega_1$, if $j\neq h(i)$, $B_i\smallsetminus A_j$ is stationary. Because $\overline{E}\not\leq^*\overline{D}$, the set
		$$S':=\{\alpha\in\omega_1\;|\;\exists i(\alpha\in E_i\smallsetminus D_{h(i)})\}$$
		is stationary (see Remark \ref{RemEqLeqStar}).
		
		Now we proceed as in the proof of Theorem \ref{ProjStat}: Let $\mathcal{H}$ be an algebra on $H(\Theta)$ such that any $M\in[H(\Theta)]^{<\omega_1}$ with $M\prec\mathcal{H}$ is in $C$. Find a Laver-style tree $T\leq(\omega_2)^{<\omega}$ (possibly with a stem) and $\alpha\in S'$ such that for any branch $b$ of $T$, $\Hull^{\mathcal{H}}(b\cup\alpha)\cap\omega_1=\alpha$. Let $i\in\omega_1$ be such that $\alpha\in E_i\smallsetminus D_{h(i)}$ and let $j$ be such that $\alpha\in D_j$. Then $j\neq h(i)$, so $B_i\smallsetminus A_j$ is stationary. Ergo we can find $\gamma\in B_i\smallsetminus A_j$ which is a ``limit point of the tree $T$'' (as in the proof of Theorem \ref{ProjStat}). Thus there is $M\prec\mathcal{H}$ such that $M\cap\omega_1\in E_i\cap D_j$ and $\sup(M\cap\omega_2)\in B_i\smallsetminus A_j$. Let $(D_n)_{n\in\omega}$ enumerate all open dense subsets of $\dP$ lying in $M$, let $(p_n)_{n\in\omega}$ be a descending sequence of elements of $\dP$ with $p_n\in D_n$ and let $p:=\bigcup_np_n\cup\{(M\cap\omega_1,M)\}$. It follows that $p\in\dP$ (since $M\in S$ and $M=\bigcup_np_n(\dom(p_n)-1)$) and $p$ forces $\dot{g}(M\cap\omega_1)=\sup(M\cap\omega_2)$, so $p$ forces $\dot{g}[\check{D}_j]\not\subseteq \check{A}_j$, a contradiction.
		
		Now let $(D_i)_{i\in\omega_1}$ be defined by $D_i:=\{p\in\dP\;|\;i\in\dom(p)\}$. It follows that there exists a filter $G$ on $\dP$ intersecting every $D_i$ and that the union of this filter is as required.
	\end{proof}
	
	Fuchs asked in \cite{FuchsCanonicalFragments} whether there is an implication from $\SRP(\mathcal{S}((D_i)_{i\in\omega_1},\kappa,\Theta))$ to $F^+((E_i)_{i\in\omega_1},\omega_2)$, where $(D_i)_{i\in\omega_1}$ is the trivial partition given by $D_0:=\omega_1$ and $D_i:=\emptyset$ for $i\neq 0$ and $(E_i)_{i\in\omega_1}$ is a maximal partition of $\omega_1$ into stationary sets, i.e. every $E_i$ is stationary and for every stationary $A\subseteq\omega_1$ there is $i\in\omega_1$ such that $A\cap E_i$ is stationary. We can answer this question in the negative in a strong way by showing the following:
	
	\begin{mylem}\label{PartNotEqu}
		Let $(D_i)_{i\in\omega_1}$ be the trivial partition of $\omega_1$ and let $(E_i)_{i\in\omega_1}$ be a partition of $\omega_1$ into at least two stationary sets. Then $\overline{D}\not\leq^*\overline{E}$.
	\end{mylem}
	
	\begin{proof}
		Let $h\colon\omega_1\to\omega_1$ be any function. Let $i:=h(0)$. So there is a club $C$ such that $D_0\cap C\subseteq E_{h(0)}$. But this is a contradiction as $E_{h(0)}$ has a stationary complement.
	\end{proof}
	
	Thus we have the following:
	
	\begin{mycol}
		Assume we are in the standard model of Martin's Maximum. Assume $(D_i)_{i\in\omega_1}$ is the trivial partition of $\omega_1$ and $(E_i)_{i\in\omega_1}$ is a partition of $\omega_1$ such that at least two $E_i, E_{i'}$ are stationary. Then there is a forcing extension where $\SRP(\mathcal{S}((D_i)_{i\in\omega_1},\omega_2,\Theta))$ holds and $F^+((E_i)_{i\in\omega_1},\omega_2)$ fails.
	\end{mycol}
	
	\begin{proof}
		Simply force with $\dP_{F^+((E_i)_{i\in\omega_1})}(\omega_2)$ to obtain $V[G]$. In the resulting model, let $F:=\bigcup G$ and $A_i:=\{\alpha\in\omega_2\;|\;F(\alpha,i)=1\}$. Then $\MM_{F^+((E_i)_{i\in\omega_1})}((A_i)_{i\in\omega_1})$ holds (so $F^+((E_i)_{i\in\omega_1},\omega_2)$ fails), and $(A_i)_{i\in\omega_1}$ is a disjoint sequence of stationary subsets of $E_{\omega}^{\omega_2}$. By Theorem \ref{DistFPlusSRP} and Lemma \ref{PartNotEqu}, $\SRP(\mathcal{S}((D_i)_{i\in\omega_1},\omega_2,\Theta))$ holds. 
	\end{proof}
	
	More subtle separations of instances of $F^+$ and $\SRP$ are also possible: E.g. it is consistent that there are two maximal partitions $(E_i)_{i\in\omega_1}$ and $(D_i)_{i\in\omega_1}$ of $\omega_1$ into stationary sets such that $\SRP(\mathcal{S}((D_i)_{i\in\omega_1},\omega_2,\Theta))$ holds for all $\Theta\geq\omega_2$ and $F^+((E_i)_{i\in\omega_1},\omega_2)$ fails: Assume we are in the standard model of $\MM$. Let $(F_i)_{i\in\omega_1}$ be a maximal partition of $\omega_1$ into stationary sets which can be constructed as follows (see \cite{FuchsCanonicalFragments}, Remark 3.17): Let $(G_i)_{i\in\omega_1}$ be any partition of $\omega_1$ such that each $G_i$ is stationary. Let $f\colon\omega_1\to\omega_1$ be defined as follows: If $j\in G_i$ for $i<j$, let $f(j):=i$, otherwise let $f(j):=0$. Ergo $f$ is regressive. For any $i\in\omega_1$, let $F_i:=f^{-1}[\{i\}]$. As $F_i\supseteq G_i\smallsetminus(i+1)$, each $F_i$ is stationary. Moreover, whenever $S$ is stationary, $f\uhr T$ is constant for some stationary $T\subseteq S$, say with value $i_0$, so $T\cap F_{i_0}$ is stationary. Now let $\sigma\colon\omega_1\times\omega_1\to\omega_1$ be a bijection and let
	$$D_i:=\bigcup_{j\in\omega_1}F_{\sigma(i,j)}\text{ and }E_i:=\bigcup_{j\in\omega_1}F_{\sigma(j,i)}$$
	Then $(D_i)_{i\in\omega_1}$ and $(E_i)_{i\in\omega_1}$ are maximal partitions of $\omega_1$ into stationary sets and for every $i,j\in\omega_1$, $D_i\cap E_j$ is stationary. Consequently:
	
	\begin{myclaim}
		$\overline{D}$ and $\overline{E}$ are incomparable in $\leq^*$.
	\end{myclaim}
	
	\begin{proof}
		Let $h,C$ witness $\overline{D}\leq^*\overline{E}$. Let $j\neq h(0)$. Then there is $\alpha\in D_0\cap E_j\cap C$. But then $\alpha$ witnesses that $D_0\cap C\not\subseteq E_{h(0)}$.
	\end{proof}
	
	Ergo we can find separate forcing extensions in which $\SRP(\mathcal{S}((D_i)_{i\in\omega_1},\omega_2,\Theta))\wedge\neg F^+((E_i)_{i\in\omega_1},\omega_2)$ and $\SRP(\mathcal{S}((E_i)_{i\in\omega_1},\omega_2,\Theta))\wedge\neg F^+((D_i)_{i\in\omega_1},\omega_2)$ hold.
	
	\section{$F$ and $F^+$ at large cardinals}
	
	In this section we investigate the connections between $F$, $F^+$ and large cardinal properties, focusing in particular on the points at which $F$ and $F^+$ can hold for the first time. In the case of $F$, this of course determines the possible patterns perfectly since $\neg F(\lambda,\kappa)$ implies $\neg F(\lambda,\kappa')$ whenever $\kappa'<\kappa$.
	
	For simplicity, we will focus on the notationally easiest cases, i.e. $F(2,\delta)$ and $F^+((D_i)_{i\in\omega_1},\delta)$, where $D_0=\omega_1$, $D_i=\emptyset$ for $i\neq0$ (it will be straightforward to see that the proofs also work for more complicated instances of the properties). For simplicity we denote those properties by $F(\delta)$ and $F^+(\delta)$ respectively. Due to this, we also let $\dP_F(\delta):=\dP_{F(2)}(\delta)$ and let $\dP_{F^+}(\delta)$ consist of functions $p\colon\gamma\to2$ for $\gamma<\delta$ such that $p(\alpha)=1$ implies that $\alpha$ has countable cofinality and $p$ is not constant on any closed copy of $\omega_1$.
	
	We also fix the following definition:
	
	\begin{mydef}
		Let $\delta$ be an ordinal. $f\colon\delta\to 2$ is an $F$-function if $f$ is not constant on any closed copy of $\omega_1$. $A\subseteq E_{\omega}^{\delta}$ is an $F^+$-set if it is stationary and does not contain a closed copy of $\omega_1$.
	\end{mydef}
	
	We start by proving some connections between the Friedman properties and large cardinal notions in $\ZFC$. Later on we will prove corresponding independence statements. Interestingly, the differences between $F$ and $F^+$ will show up twice: On one hand, the complexity of the statement ``$A$ is stationary'' leads us to weaker $\ZFC$ results while the stronger properties of the forcing $\dP_{F^+}$ as opposed to the forcing $\dP_F$ allows us to obtain stronger independence results, so that we always obtain a complete characterization.
	
	\begin{mysen}
		Assume $\kappa$ is weakly compact and $F(\delta)$ fails for all regular cardinals $\delta<\kappa$. Then $F(\kappa)$ fails.
	\end{mysen}
	
	\begin{proof}
		For any regular $\delta<\kappa$, let $f_{\delta}\colon\delta\to 2$ be a function witnessing the failure of $F(\delta)$. Let $T$ be the tree generated by $f_{\delta}$ for $\delta<\kappa$. By weak compactness there is a branch $b\colon\kappa\to 2$ such that $b\uhr\alpha\in T$ for any $\alpha<\kappa$.
		
		\begin{myclaim}
			$b$ witnesses that $F(\kappa)$ fails.
		\end{myclaim}
		
		\begin{proof}
			Otherwise let $c\colon\omega_1\to\kappa$ be such that $b\uhr c[\omega_1]$ is constant. Let $\sup(c[\omega_1])=:\alpha<\kappa$. There is a regular $\delta<\kappa$ such that $b\uhr\alpha=f_{\delta}\uhr\alpha$. However, then $c$ witnesses that $f_{\delta}$ does not witness $\neg F(\delta)$, a contradiction.
		\end{proof}
		
		Ergo $F(\kappa)$ fails.
	\end{proof}

	For embedding-based large cardinal notions, we can observe that the failure persists until one cardinal above the target of the critical point using a well-known observation:
	
	\begin{mysen}
		Assume $\kappa\leq\lambda$ are cardinals such that $\kappa$ is $\lambda$-supercompact. Assume $F(\delta)$ fails for all regular cardinals $\delta<\kappa$. Then $F(\lambda^+)$ fails.
	\end{mysen}
	
	\begin{proof}
		Let $j\colon V\to M$ be a $\lambda$-supercompact embedding, i.e. ${}^{\lambda}M\subseteq M$ and $j(\kappa)>\lambda$. As ${}^{\lambda}M\subseteq M$, $(\lambda^+)^M=\lambda^+$ and since $j(\kappa)$ is inaccessible in $M$, $j(\kappa)>\lambda^+$. So $(F(\lambda^+))^M$ fails, witnessed by some $f\colon\lambda^+\to 2$, $f\in M\subseteq V$. Assume $c\subseteq\lambda^+$ is closed and has ordertype $\omega_1$ such that $f\uhr c$ is constant. Then $c\in M$ by the closure of $M$, a contradiction.
	\end{proof}
	
	For $F^+$, less is provable even regarding the first point of failure due to the added complexity of ``$A$ is stationary''. We first prove the following $\ZFC$ result:
	
	\begin{mysen}
		Assume $\kappa\leq\lambda$ are cardinals such that $\kappa$ is $\lambda$-supercompact. Assume $F^+(\delta)$ fails for all regular cardinals $\delta<\kappa$. Then $F^+(\lambda)$ fails.
	\end{mysen}
	
	\begin{proof}
		Let $j\colon V\to M$ be a $\lambda$-supercompact embedding, i.e. ${}^{\lambda}M\subseteq M$ and $j(\kappa)>\lambda$. Then $(F^+(\lambda))^M$ fails, witnessed by some $A\subseteq\lambda$. In $M$, $A$ is a stationary subset of $E_{\omega}^{\lambda}$, which implies that $A$ is a stationary subset of $E_{\omega}^{\lambda}$ in $V$ by the closure of $M$ (any club witnessing the nonstationarity of $A$ in $V$ would be in $M$). Ergo $A$ witnesses that $F^+(\lambda)$ fails in $V$, since any closed $c\subseteq A$ with ordertype $\omega_1$ would be in $M$.
	\end{proof}
	
	Now we present related consistency results which show that our previous theorems cannot be improved (in various directions).
	
	We will later require complicated lifting arguments which involve the ``filling'' of partial Friedman sets. The following theorem serves as a gentle introduction into the machinery:
	
	\begin{mysen}\label{ThmSuperCompF}
		Assume $F^+(\delta)$ holds for all regular $\delta\geq\omega_2$ and $\GCH$ holds above $\omega_2$. Let $\kappa\leq\lambda$ be cardinals such that $\kappa$ is $\lambda^+$-supercompact. There is a forcing extension where $\kappa$ is $\lambda$-supercompact, $F(\delta)$ fails for all regular $\delta\in[\omega_2,\lambda^+]$ and $F^+(\lambda^{++})$ holds.
	\end{mysen}
	
	We need the following preservation result:
	
	\begin{mylem}\label{SmallPresFPlus}
		Assume $\delta$ is a regular cardinal such that $F^+(\delta)$ holds and $\dP$ is a poset with $|\dP|<\delta$ that does not collapse $\omega_1$. Then $F^+(\delta)$ holds after forcing with $\dP$.
	\end{mylem}
	
	\begin{proof}
		Assume $\dot{S}$ is a $\dP$-name for a stationary subset of $E_{\omega}^{\delta}$. For $p\in\dP$, let $S_p:=\{\alpha\in\delta\;|\;p\Vdash\check{\delta}\in\dot{S}\}$. Since $\dot{S}$ is forced to be a subset of $\bigcup_{p\in\dP}S_p$, one of the $S_p$ is stationary. Since $F^+(\delta)$ holds, there is a closed set $c\subseteq S_p$ with ordertype $\omega_1$. Ergo $p$ forces that $\check{c}$ is a closed subset of $\dot{S}$ with ordertype $\dot{\omega}_1$ (since $\dP$ does not collapse $\omega_1$).
	\end{proof}
	
	\begin{proof}[Proof of Theorem \ref{ThmSuperCompF}]
		Let $(\dP_{\alpha},\dot{\dQ}_{\alpha})_{\alpha\leq\lambda^+}$ be an Easton-support iteration such that $\dot{\dQ}_{\alpha}$ is a $\dP_{\alpha}$-name for $\dP_F(\alpha)$ whenever $\alpha$ is regular and a $\dP_{\alpha}$-name for the trivial poset otherwise. Let $\dP:=\dP_{\lambda^++1}:=\dP_{\lambda^+}*\dot{\dQ}_{\lambda^+}$.
		
		\begin{myclaim}
			For any $\alpha\leq\lambda^++1$, the following holds:
			\begin{enumerate}
				\item $\dP_{\alpha}$ forces that $F(\delta)$ fails for all regular $\delta<\alpha$.
				\item $\dP_{\alpha}$ does not collapse any cardinals.
				\item $\dP_{\alpha}$ forces $\GCH$ above $\omega_2$.
			\end{enumerate}
		\end{myclaim}
		
		\begin{proof}
			We do the proof by induction on $\alpha$.
			
			Let $\alpha:=\beta+1$. In case $\beta$ is not regular, $\dP_{\beta}$ is essentially the same as $\dP_{\alpha}$, so the statement is clear (because we also have no new requirements in (1)). Assume $\beta$ is regular. Then $\dP_{\alpha}=\dP_{\beta}*\dP_F(\check{\beta})$. By the induction hypothesis (1), in $V[\dP_{\beta}]$, $\dP_F(\beta)$ is ${<}\,\beta$-strategically closed and has size $\beta$ (by the induction hypothesis (2)). Ergo it forces $\neg F(\beta)$, does not collapse any cardinals and does not increase powerset sizes (by a nice name argument). This shows that all three statements hold for $\alpha$.
			
			Now assume $\alpha$ is a limit ordinal. Since $|\dP_{\alpha}|\leq|\alpha|$, it suffices to show that for any cardinal $\delta<\alpha$, $\dP_{\alpha}$ forces $2^{\delta}=\delta^+$, does not collapse $\delta$ and forces $\neg F(\delta)$ if $\delta$ is regular. So let $\delta<\alpha$. Then $\dP_{\alpha}$ decomposes as $\dP_{\delta+1}*\dP_{\delta+1,\alpha}$, where a similar analysis as before and the usage of Easton support shows that $\dP_{\delta+1,\alpha}$ is $\delta+1$-strategically closed in $V[\dP_{\delta+1}]$. Ergo any subset of $\delta$ added by $\dP_{\alpha}$ has been added by $\dP_{\delta+1}$. Then the induction hypothesis for $\delta+1$ shows that the required statements hold.
		\end{proof}
		
		Lastly, we show the following:
		
		\begin{myclaim}
			For each $\alpha\leq\lambda^++1$ that is not a singular limit cardinal, $\dP_{\alpha}$ has a dense subset of cardinality $\leq|\alpha|$.
		\end{myclaim}
		
		\begin{proof}
			We define the set inductively. Assume $\dR_{\beta}\subseteq\dP_{\beta}$ has been defined for each $\beta<\alpha$.
			
			If $\alpha=\beta+1$ for $\beta$ nonregular, we simply let $\dR_{\beta+1}$ consist of $(p,\emptyset)$ for $p\in\dR_{\beta}$. If $\beta$ is regular, we let $\dR_{\beta+1}$ consist of $(p,\tau)$, where $p\in\dR_{\beta}$ and $\tau$ is a nice $\dR_{\beta}$-name for an element of $\dot{\dQ}_{\beta}$. It follows from $\GCH$ above $\omega_2$, the size of $\dot{\dQ}_{\beta}$ being forced to equal $\beta$ and the fact that $|\dR_{\beta}|=\beta$ that $|\dR_{\beta+1}|=\beta$ and it is clearly dense.
			
			If $\alpha$ is a limit, we let $\dR_{\alpha}$ consist of those $f\in\dP_{\alpha}$ such that $f\uhr\beta\in\dR_{\beta}$ for each $\beta<\alpha$. We see as follows that $\dR_{\alpha}$ is dense: Given any $p\in\dP_{\alpha}$, let $q$ be defined as follows: If $q\uhr\beta$ has been constructed, let $q(\beta)$ be such that $(q\uhr\beta,q(\beta))\in\dR_{\beta+1}$ and $q\uhr\beta\Vdash q(\beta)=p(\beta)$ (this is possible by the definition of $\dR_{\beta+1}$). This also shows that in limit steps, $q\uhr\beta$ will be in $\dR_{\beta}$, so $q\in\dR_{\alpha}$.
			
			Clearly, if $\alpha$ is not a limit cardinal, $|\dR_{\alpha}|=|\dR_{\beta+1}|$ where $\beta$ is the largest cardinal below $\alpha$, so $|\dR_{\alpha}|=\beta=|\alpha|$. If $\alpha$ is a regular limit cardinal, we can identify $\dR_{\alpha}$ with the set of all ${<}\,\alpha$-sized partial functions from $\alpha$ to $\alpha$, so $|\dR_{\alpha}|=\alpha$ by $\GCH$ above $\omega_2$.
		\end{proof}
		
		So $F^+(\lambda^{++})$ holds after forcing with $\dP$ by Lemma \ref{SmallPresFPlus}. The only thing left is to show that $\dP$ forces that $\kappa$ is $\lambda$-supercompact. To this end, let $G$ be $\dP$-generic and let $j\colon V\to M$ be a $\lambda^+$-supercompact embedding, i.e. ${}^{\lambda^+}M\subseteq M$ and $j(\kappa)>\lambda^+$. For any $\alpha<\beta\leq\lambda^++1$, let $G\uhr\alpha$ be the $\dP_{\alpha}$-generic filter induced by $G$, let $G\uhr[\alpha,\beta)$ be the $\dP_{\alpha,\beta}^{G\uhr\alpha}$-generic filter induced by $G$ (where we write $\dP_{\beta}\cong\dP_{\alpha}*\dP_{\alpha,\beta}$) and let $G(\alpha)$ be the $\dot{\dQ}_{\alpha}^{G\uhr\alpha}$-generic filter induced by $G$. We aim to find a master condition $q\in j(\dP)$ which allows us to lift the embedding to $j\colon V[G]\to M[H]$ inside $V[H]$ for some $j(\dP)$-generic filter $H$.
		
		\begin{myclaim}
			For any $\alpha\in[\kappa,\lambda^++1]$ there is a condition $q_{\alpha}\in j(\dP)\uhr[\lambda^++1,j(\alpha))$ such that whenever $H$ is $j(\dP)\uhr[\lambda^++1,j(\alpha))$-generic containing $q_{\alpha}$, $j[G\uhr\alpha]\subseteq G*H$.
		\end{myclaim}
		
		\begin{proof}
			We do the proof by induction. First we let $q_{\kappa}:=\emptyset$ which works as $j[G\uhr\kappa]=G\uhr\kappa$.
			
			Now assume $q_{\beta}$ has been defined and $\alpha=\beta+1$. If $\beta$ is not a regular cardinal, we are done, so assume $\beta$ is regular. Let $f:=\bigcup G(\beta)$. Then $f$ is not constant on any closed copy of $\omega_1$ and the same holds for $j[f]$. However, $j[f]$ is only defined on $j[\beta]$ and so $f$ is not an $F$-function. Furthermore, simply patching $f$ by letting it be $0$ or $1$ on $\sup(j[\beta])\smallsetminus j[\beta]$ does not work because $j[\beta]$ is disjoint from $j(\kappa)\smallsetminus\kappa$ which clearly contains a closed set with ordertype $\omega_1$.
			
			Let $H$ be any $j(\dP)\uhr[\lambda^++1,j(\beta))$-generic filter containing $q_{\beta}$ and work in $M[G*H]$. We note that $j[f]\in M[G*H]$ because $G(\beta)\in M[G*H]$, $\bigcup G(\beta)$ is a function on $\beta\leq\lambda^+$ and $j\uhr\lambda^+\in M[G*H]$. Let $\eta:=\sup(j[\beta])$. By regularity of $j(\beta)$ (and $\beta\leq\lambda^+$ which implies $\beta<j(\beta)$), $\eta<j(\beta)$. Furthermore, $F(\delta)$ fails for each regular $\delta<j(\beta)$ (we are in an extension by $j(\dP)\uhr j(\beta)$) so there exists an $F$-function $g\colon\eta\to 2$ in $M[G*H]$ by Lemma \ref{FriedmanPumpUp} (either $j(\beta)$ is a limit cardinal in which case it is clear that $g$ exists or $j(\beta)=j(\beta^{-})$ in which case $|\eta|=\beta^-$ in $M[G*H]$ and $F(\beta^-)$ fails). Let $h:=j[f]\cup(g\uhr(\eta\smallsetminus j[\beta]))$. We claim that $h$ is an $F$-function. To this end, let $c\subseteq\eta$ be closed with ordertype $\omega_1$. Assume without loss of generality that any $\delta\in c$ has cofinality $\omega$ and let $\xi:=\sup(c)$.
			
			\begin{itemize}
				\item[Case 1:] $j[\beta]$ is unbounded in $\xi$. In this case, $j[\beta]$ contains a club in $\xi$ (because $j$ is continuous at ordinals of cofinality ${<}\,\kappa$ and $\cf(\xi)=\omega_1$). Thus $c\cap j[\beta]$ contains a club $d$ in $\xi$ (necessarily with ordertype $\omega_1$). Let $e:=j^{-1}[d]$. Then $e$ is closed (by continuity of $j$), $\otp(e)=\omega_1$ and $e\subseteq\beta$. Ergo $f(\zeta_0)=0$ and $f(\zeta_1)=1$ for some $\zeta_0,\zeta_1\in e$. But this implies $j(\zeta_0),j(\zeta_1)\in d$ and
				$$h(j(\zeta_0))=j(f)(j(\zeta_0))=j(f(\zeta_0))=j(0)=0$$
				and similarly $j(j(\zeta_1))=1$.
				\item[Case 2:] $j[\beta]$ is bounded in $\xi$. In this case we can assume without loss of generality that $c$ is disjoint from $j[\beta]$ (simply take a tail segment). However, this implies that $g$ is constant on $c$, a contradiction.
			\end{itemize}
			
			Ergo in $M[G*H]$ the function $h$ is a condition in $\dP(j(\beta))$. It follows that $q_{\beta+1}:=q_{\beta}*\dot{h}$ is as required (letting $\dot{h}$ be a name for $h$): Whenever $H$ is $j(\dP)\uhr[\lambda^++1,j(\beta+1))$-generic containing $q_{\beta+1}$, $H\uhr j(\beta)$ contains $q_{\beta}$. Ergo $j[G\uhr\beta]\subseteq G*(H\uhr j(\beta))$ and $j[G\uhr\beta+1]\subseteq G*H$, since whenever $p\in G\uhr\beta+1$, $j(p(\beta))$ is a subset of $f$ and thus of $H(j(\beta))$.
			
			In case $\alpha$ is a limit we can simply let $q_{\alpha}$ be the limit of the conditions constructed thus far. For any $\alpha\in[\kappa+1,\lambda^++1]$, $j[\alpha]$ is in $M$, so its supremum can never be a regular cardinal, since $\sup(j[\alpha])\geq j(\kappa)>\lambda^+$ and $\cf(\sup(j[\alpha]))=\cf(\alpha)\leq\lambda^+$. Since $\dom(q_{\alpha})\subseteq[\lambda^++1,j(\alpha))\cap\im(j)$, it is therefore an Easton set. Clearly, $q_{\alpha}$ is as required.
		\end{proof}
		
		Now let $H$ be a $j(\dP)\uhr[\lambda^++1,j(\lambda^++1))$-generic filter containing $q_{\lambda^++1}$. It follows that in $V[G*H]$ we can lift $j$ to $j^+\colon V[G]\to M[G*H]$, so in $V[G*H]$ there exists a $\kappa$-complete normal and fine ultrafilter over $(\mathcal{P}([\lambda]^{<\kappa}))^{V[G]}$. This ultrafilter has size $2^{(\lambda^{<\kappa})}=\lambda^+$ by assumption so it has not been added by the ${<}\,\lambda^{++}$-strategically closed tail forcing. Ergo $\kappa$ is $\lambda$-supercompact in $V[G]$.
	\end{proof}
	
	In the case we want the ``split'' to occur at limit cardinals we have to use a slightly more involved argumentation, since the iteration to force the failure of $F$ at every cardinal $\delta<\kappa$ (for some limit cardinal $\kappa$) will have size at least $\kappa$. In those cases, we have to strengthen our hypothesis from $F^+$ globally to being in a standard model of $\MM$.
	
	We will also need the following preservation Lemma:
	
	\begin{mylem}\label{PresLemma}
		Let $\omega_2\leq\delta<\kappa$ be regular cardinals. Let $A\subseteq E_{\omega}^{\delta}$ and $B\subseteq E_{\omega}^{\kappa}$ be stationary sets such that neither of them contains a closed copy of $\omega_1$. Then $A$ does not contain a closed copy of $\omega_1$ after forcing with $\dP(B)$.
	\end{mylem}
	
	\begin{proof}
		Assume toward a contradiction that a condition $p\in\dP(B)$ forces that there is a closed copy of $\omega_1$ contained in $E_{\omega}^{\delta}$. Ergo (after strengthening $p$ if necessary) there is $\gamma\leq\delta$ such that $p$ forces $A\cap\gamma$ to contain a club $\dot{C}$.
		
		By \cite{ForemanMagidorMutStat}, Theorem 7, there is $M\prec(H(\Theta),\dot{C},A,B)$ with $p\in M$ such that $\sup(M\cap\gamma)\notin A$ (as $A\cap\gamma$ is costationary in $\gamma$ in $V$) and $\sup(M\cap\kappa)\in B$. However, because $\sup(M\cap\kappa)\in B$ we can find a condition $q\leq p$ that lies in any open dense subset of $\dP(B)$ that is in $M$. Ergo $q$ forces $\sup(M\cap\gamma)\in\dot{C}$, a contradiction.
	\end{proof}
	
	Now we can show our incompactness results for the failure of $F$ and $F^+$:
	
	\begin{mysen}
		Let $\lambda$ be supercompact and $\kappa>\lambda$ a Mahlo cardinal. There is a forcing extension where $\lambda=\omega_2$, $\kappa$ remains Mahlo, $F(\delta)$ fails for all regular $\delta<\kappa$ and $F^+(\kappa)$ holds.
	\end{mysen}
	
	\begin{proof}
		We assume without loss of generality that $\GCH$ holds, otherwise we can achieve this by taking a forcing extension before the following construction.
		
		Let $l\colon\lambda\to V_{\lambda}$ be a Laver function for $\lambda$. Let $\dP_{MM}(\lambda)$ be the standard iteration forcing $\MM$ defined from $l$ described in the proof of Theorem \ref{MMFLambda}.
		
		In $V[\dP_{MM}(\lambda)]$, let $\dP:=(\dP_{\alpha},\dot{\dQ}_{\alpha})_{\alpha<\kappa}$ be an Easton support iteration where $\dot{\dQ}_{\alpha}$ is a $\dP_{\alpha}$-name for $\dP_F(\alpha)$ if $\alpha\geq\omega_2$ is regular and for the trivial poset otherwise. Clearly $\dP_{MM}(\lambda)*\dP$ is $\kappa$-cc.\ and it follows as before that $\dP_{MM}(\lambda)*\dP$ does not change cofinalities and powerset sizes above $\lambda$ (it forces $2^{\omega}=2^{\omega_1}=\omega_2=\lambda$), so $\kappa$ remains Mahlo in the extension. Additionally, $\dP_{MM}(\lambda)*\dP$ forces $\neg F(\delta)$ for all regular $\delta<\kappa$.
		
		All that is left is to show that $F^+(\kappa)$ holds in $V[\dP_{MM}(\lambda)*\dP]$. To this end, fix a $\dP_{MM}(\lambda)*\dP$-name $\dot{A}$ for a stationary subset of $E_{\omega}^{\kappa}$. Thanks to the properties of $\dP_{MM}(\lambda)$, there is a $\kappa$-supercompact embedding $j\colon V\to M$ for $\lambda$ such that $j(\dP_{MM}(\lambda))\cong \dP_{MM}(\lambda)*\dP*\dP(\dot{A})*\dot{\Coll}(\omega_1,\check{\kappa})*\dot{\dR}$ where $\dot{\dR}$ is forced to be semiproper.
		
		We now argue similarly to the proof of Theorem \ref{MMFLambda}. Let $G*H$ be any $\dP_{MM}(\lambda)*\dP$-generic filter. In $V[G*H]$, for $\alpha<\beta\leq\kappa$, let $H\uhr\alpha$ be the $\dP_{\alpha}$-generic filter induced by $H$, let $H\uhr[\alpha,\beta)$ be the $\dP_{\alpha,\beta}^{G*(H\uhr\alpha)}$-generic filter induced by $H$ (where $\dP_{\beta}\cong\dP_{\alpha}*\dP_{\alpha,\beta}$) and let $G(\alpha)$ be the $\dot{\dQ}_{\alpha}^{G*(H\uhr\alpha)}$-generic filter induced by $G$. We aim to build a master condition which allows us to lift $j$ and note that
		$$j(\dP_{MM}(\lambda)*\dP)\cong j(\dP_{MM}(\lambda))*j(\dP)\cong\dP_{MM}(\lambda)*\dP*\dP(\dot{A})*\dot{\Coll}(\omega_1,\check{\kappa})*\dot{\dR}*j(\dP)$$
		First of all, let $G^{\text{lift}}$ be any $j(\dP_{MM}(\lambda))$-generic filter containing $G*H$. Work in $V[G^{\text{lift}}]$.
		
		\begin{myclaim}
			For any $\alpha\in[\lambda,\kappa]$ there is a condition $q_{\alpha}\in j(\dP)^{G^{\text{lift}}}\uhr\alpha$ such that whenever $I$ is $j(\dP)^{G^{\text{lift}}}\uhr\alpha$-generic containing $q_{\alpha}$, $j[G*H\uhr\alpha]\subseteq G^{\text{lift}}*I$.
		\end{myclaim}
		
		\begin{proof}
			We do the proof by induction. First we left $q_{\lambda}:=\emptyset$ which works because $j(\dP)^{G^{\text{lift}}}\uhr j(\lambda)$ is the trivial forcing (since $j(\dP_{MM})$ forces $j(\lambda)=\omega_2$). Now assume $q_{\beta}$ has been defined and $\alpha=\beta+1$. If $\beta$ is not a regular cardinal, we are done, so assume $\beta$ is regular. Let $f:=\bigcup H(\beta)$.
			
			\begin{mysclaim}
				In $V[G^{\text{lift}}]$, $f$ is not constant on any closed copy of $\omega_1$.
			\end{mysclaim}
			
			\begin{proof}
				The statement is clearly true in $V[G*(H\uhr\beta+1)]$, so we only are left to show that it remains true in $V[G^{\text{lift}}]$. First of all, $V[G*H]$ contains no $\omega_1$-length sequences not contained in $V[G*(H\uhr\beta+1)]$, so the statement is still true in $V[G*H]$. By Lemma \ref{PresLemma}, this is preserved by $\dP(\dot{A}^{G*H})$. It follows as in the proof of Theorem \ref{MMFLambda} that $f$ is not constant on any closed copy of $\omega_1$ in $V[G^{\text{lift}}]$.
			\end{proof}
			
			Let $I_{\beta}$ be $j(\dP)\uhr\beta$-generic containing $q_{\beta}$. In $V[G^{\text{lift}}*I_{\beta}]$, we can define a condition $q(\beta)$ extending $j[f]$ as in the proof of Theorem \ref{ThmSuperCompF}. Then $q_{\beta+1}:=q_{\beta}*q(\beta)$ is easily seen to be as required, since $j[G*H\uhr\beta]\subseteq G^{\text{lift}}*I_{\beta}$ by the inductive hypothesis and $j[G*H\uhr\beta+1]$ is equal to $j[G*H\uhr\beta]$ with $j[H(\beta)]$ at point $j(\beta)$ which is contained in $I(\beta)$ since $q(\beta)\in I(\beta)$.
			
			The limit step follows as before by taking a limit of the previously defined conditions.
		\end{proof}
		
		So let $H^{\text{lift}}$ be a $j(\dP)^{G^{\text{lift}}}$-generic filter containing $q_{\kappa}$. It follows that in $V[G^{\text{lift}}*H^{\text{lift}}]$ we can lift the embedding $j$ to $j\colon V[G*H]\to M[G^{\text{lift}}*H^{\text{lift}}]$. Furthermore, in $M[G^{\text{lift}}*H^{\text{lift}}]$, $\dot{A}^{G*H}$ contains a closed copy $c$ of $\omega_1$. By the closure of $M[G^{\text{lift}}*H^{\text{lift}}]$, $j[c]$ is in $M[G^{\text{lift}}*H^{\text{lift}}]$ as well and is a closed copy of $\omega_1$ contained in $j[\dot{A}^{G*H}]\subseteq j(\dot{A}^{G*H})$. By elementarity, in $V[G*H]$, $\dot{A}^{G*H}$ contains a closed copy of $\omega_1$.
		
	\end{proof}
	
	We now turn to the principle $F^+$. As we have seen before, the proofs of the $\ZFC$ results do not adapt to $F^+$ because of the added complexity. However, this complexity also allowed us to define a more well-behaved forcing notion adding witnesses to $\neg F^+$ ($\dP_{F^+}(\delta)$ is $\delta$-strategically closed and not just ${<}\,\delta$-strategically closed) which will now grant us the ability to obtain corresponding stronger non-implications.
	
	\begin{mysen}
		Assume $\lambda$ is supercompact and $\kappa>\lambda$ is weakly compact. There is an extension where $\lambda=\omega_2$, $\kappa$ remains weakly compact, $F^+(\delta)$ fails for all regular $\delta<\kappa$ and $F^+(\kappa)$ holds.
	\end{mysen}
	
	\begin{proof}
		We again assume $\GCH$.
		
		As before, let $\dP_{MM}(\lambda)$ be the standard iteration forcing Martin's Maximum from $\lambda$. In $V[\dP_{MM}(\lambda)]$, let $\dP:=(\dP_{\alpha},\dot{\dQ}_{\alpha})_{\alpha<\kappa}$ be an Easton support iteration where $\dot{\dQ}_{\alpha}$ is a $\dP_{\alpha}$-name for $\dP_{F^+}(\alpha)$ if $\alpha\geq\omega_2$ is regular and for the trivial poset otherwise. As before, $\dP_{MM}(\lambda)*\dP$ forces that $F^+(\delta)$ fails for all regular $\delta<\kappa$ and does not change cofinalities nor powerset sizes. It also follows the same way as before that after forcing with $\dP_{MM}(\lambda)*\dP$, $F^+(\kappa)$ holds and $F^+(\delta)$ fails for all $\delta<\kappa$, so the only thing left to show is that $\kappa$ remains weakly compact in the extension.
		
		Clearly $\kappa$ remains inaccessible so we are aiming to show that it still has the tree property in $V[G]$, where $G$ is any $\dP_{MM}*\dP$-generic filter. To this end, let $\dot{T}$ be a $\dP_{MM}(\lambda)*\dP$-name for a tree on $\kappa$. By inaccessibility of $\kappa$ in $V$, let $M\prec(H(\Theta),\in)$ contain all relevant information with $|M|=\kappa$ and $M^{<\kappa}\subseteq M$. By the weak compactness of $\kappa$ in $V$, let $j\colon M\to N$ be an elementary embedding with critical point $\kappa$ such that $N^{<\kappa}\subseteq N$. It follows that
		$$j(\dP_{MM}*\dP)\cong\dP_{MM}*\dP*(j(\dP)^{N[G]}\uhr[\kappa,j(\kappa)))$$
		\begin{myclaim}
			In $V[G]$, $j(\dP)^{N[G]}\uhr[\kappa,j(\kappa))$ is $\kappa$-strategically closed.
		\end{myclaim}
		
		\begin{proof}
			In $N[G]$, $j(\dP)^{N[G]}\uhr[\kappa,j(\kappa))$ is $\kappa$-strategically closed (simply by its definition, since $N$ models enough $\ZFC$). In $V$, $N^{<\kappa}\subseteq N$. Ergo in $V[G]$, $N[G]^{<\kappa}\subseteq N[G]$ because $\dP_{MM}*\dP$ is $\kappa$-cc.. Since every ${<}\,\kappa$-sequence of elements of $j(\dP)^{N[G]}\uhr[\kappa,j(\kappa))$ in $V[G]$ is in $N[G]$, the winning strategy from the perspective of $N[G]$ is actually a winning strategy.
		\end{proof}
		
		Let $G^{\text{lift}}$ be any $j(\dP_{MM}*\dP)$-generic filter extending $G$. In $V[G^{\text{lift}}]$ there is a lift of $j$ to $j\colon M[G]\to H[G^{\text{lift}}]$. Now a standard computation shows that $\dot{T}^G$ has a cofinal branch in $V[G^{\text{lift}}]$, so we are done after showing:
		
		\begin{myclaim}
			$\dot{T}^G$ has a cofinal branch in $V[G]$.
		\end{myclaim}
		
		\begin{proof}
			In $V[G]$, there is a $j(\dP)^{N[G]}\uhr[\kappa,j(\kappa))$-name $\dot{b}$ forced to be a cofinal branch of $\dot{T}^G$. Using that $j(\dP)^{N[G]}\uhr[\kappa,j(\kappa))$ is $\kappa$-strategically closed and thus in particular strongly ${<}\,\kappa$-distributive we can find a descending sequence $(p_{\alpha})_{\alpha<\kappa}$ of elements of $(j(\dP))^{N[G]}\uhr[\kappa,j(\kappa))$ and a sequence $(x_{\alpha})_{\alpha<\kappa}$ such that for any $\alpha<\kappa$, $p_{\alpha}$ forces that $\dot{b}\uhr\check{\alpha}=\check{x}_{\alpha}$. This implies that $\bigcup_{\alpha<\kappa}x_{\alpha}$ (which is in $V[G]$) is a cofinal branch of $\dot{T}^G$.
		\end{proof}
		
		This shows that $\kappa$ has the tree property in $V[G]$.
	\end{proof}
	
	\begin{mybem}
		The preceding proof (together with the observation that it does not work for $F(\kappa)$) shows that while $\kappa$-strategically closed forcings (or even strongly ${<}\,\kappa$-distributive forcings) cannot add branches to $\kappa$-trees, the same cannot be said for merely ${<}\,\kappa$-strategically closed forcings.
	\end{mybem}
	
	Lastly, we prove a similar result for embedding-based large cardinals:
	
	\begin{mysen}
		Assume $\kappa\leq\lambda$ are regular cardinals such that $\kappa$ is $\lambda$-supercompact. Assume $\GCH$ holds above $\omega_2$ as well as $F^+(\delta)$ for all regular $\delta\geq\omega_2$. There is a forcing extension where $\kappa$ remains $\lambda$-supercompact, $F^+(\delta)$ fails for all regular $\delta\leq\lambda$ and $F^+(\lambda^+)$ holds.
	\end{mysen}
	
	\begin{proof}
		Let $\dP:=(\dP_{\alpha},\dot{\dQ}_{\alpha})_{\alpha<\lambda+1}$ be an Easton support iteration where $\dot{\dQ}_{\alpha}$ is a $\dP_{\alpha}$-name for $\dP_{F^+}(\alpha)$ if $\alpha\geq\omega_2$ is regular and for the trivial poset otherwise. Then, just like before, it follows that $\dP$ preserves all cofinalities and powerset sizes and forces that $F^+(\delta)$ fails for all regular $\delta\leq\lambda$ (since the stationarity of any added set is preserved thanks to the closure of the tail forcing). Furthermore, as in the proof of Theorem \ref{ThmSuperCompF} it follows that $\dP$ forces $F^+(\lambda^+)$ to hold. The only thing left to show is that $\kappa$ remains $\lambda$-supercompact in $V[G]$, where $G$ is $\dP$-generic. To this end, let $j\colon V\to M$ be a $\lambda$-supercompact embedding. We first construct a master condition as before and then use the stronger properties of $\dP_{F^+}$ to build a sufficiently generic filter in $V[G]$ which allows us to lift the embedding.
		
		We have a similar claim to before:
		
		\begin{myclaim}
			There is a condition $q\in j(\dP)\uhr[\lambda+1,j(\lambda+1))$ such that whenever $H$ is $j(\dP)\uhr[\lambda+1,j(\lambda+1))$-generic containing $q$, $j[G]\subseteq G*H$.
		\end{myclaim}
		
		\begin{proof}
			We give a sketch of the proof. Everything works basically the same as before, we just have to be a bit more careful due to the added requirement of stationarity. Again let $G\uhr\alpha$ be the $\dP_{\alpha}$-generic filter induced by $G$, let $G(\alpha)$ let the $\dot{\dQ}_{\alpha}^{G\uhr\alpha}$-generic filter induced by $G$ and let $G\uhr[\alpha,\beta)$ be the $\dP_{\alpha,\beta}$-generic filter induced by $G$. Let $F_{\alpha}:=\bigcup G(\alpha)$ and let $A_{\alpha}:=\{\gamma\in\alpha\;|\;F(\gamma)=1\}$. Then $A_{\alpha}$ is a stationary subset of $\alpha$ that does not contain a closed copy of $\omega_1$.
			
			Furthermore, $j[A_{\alpha}]\in M[G]$ (it is definable from $G(\alpha)$ and $j\uhr\lambda$) and it is a stationary subset of $\sup(j[\alpha])$ there: Otherwise there would be $C\subseteq\sup(j[\alpha])\in M[G]$ club in $\sup(j[\alpha])$ such that $C\cap j[A_{\alpha}]=\emptyset$. Let $D:=j^{-1}[C]\in V[G]$ (note we do not need a lift of $j$ here since we take the pointwise preimage). Then $D$ is a ${<}\,\kappa$-club in $\alpha$ in $V[G]$ (i.e. it is unbounded and closed under limits of cofinality ${<}\,\kappa$), so its closure under limits of cofinality $\geq\kappa$ is club in $\alpha$ and intersects $A_{\alpha}$ nonemptily. However, any such ordinal is already in $D$ (since $A_{\alpha}$ consists of points of countable cofinality) and thus $C\cap j[A_{\alpha}]$ is nonempty.
			
			Again by the closure of $j(\dP)\uhr[\lambda+1,j(\lambda)+1)$, $j[A_{\alpha}]$ is a stationary subset of $\sup(j[\alpha])$ consisting of ordinals of countable cofinality that does not contain a closed copy of $\omega_1$. So its characteristic function is a condition in $(\dot{\dQ}_{j(\alpha)})^{M[G*(H\uhr\alpha)]}$. We note that unlike the previous cases, working with $\dP_F$, here we simply take the pointwise image of the generic function added by $G(\alpha)$ and ``fill it up'' with $0$ wherever it is not defined.
			
			Now it should be clear how to proceed, given the techniques we have used before.
		\end{proof}
		
		$\dP$ contains a dense set $\dR$ of cardinality $\lambda$ and thus has the $\lambda^+$-cc.. Ergo, there are at most $\lambda^+$ many maximal antichains of $\dP$ contained in $\dR$. By elementarity, there are at most $j(\lambda^+)$ many maximal antichains of $j(\dP)$ contained in $j(\dR)$ lying in $M$. However, $j(\lambda^+)$ is equal to the set of all equivalence classes of functions from $\lambda^{<\kappa}$ to $\lambda$, so $|j(\lambda^+)|=2^{\lambda}=\lambda^+$ in $V$ (which is of course not changed by forcing extensions). Ergo in $V[G]$ we can find an enumeration $(A_{\alpha})_{\alpha<\lambda^+}$ of all maximal antichains of $j(\dP)\uhr[\lambda+1,j(\lambda+1))$ lying in $M[G]$. Using the fact that $j(\dP)\uhr[\lambda+1,j(\lambda+1))$ is $\lambda^+$-strategically closed we can now find a descending sequence $(q_{\alpha})_{\alpha<\lambda^+}$ in $V[G]$ of elements of $j(\dP)\uhr[\lambda+1,j(\lambda+1))$ such that $q_0\leq q$ and for any $\alpha<\lambda^+$, $q_{\alpha}$ is below an element of $A_{\alpha}$. Ergo the upward closure of $\{q_{\alpha}\;|\;\alpha<\lambda^+\}$, denote it by $H$, is a filter for $j(\dP)\uhr[\lambda+1,j(\lambda+1))$ that is generic over $M[G]$ (by density of $j(\dR)$). Furthermore, $H\in V[G]$. Ergo, in $V[G]$ we can lift the embedding $j$ to $j\colon V[G]\to M[G*H]$ which shows that $\kappa$ is $\lambda$-supercompact in $V[G]$.
	\end{proof}
	
	\begin{mybem}
		As before we can see that, while $\lambda^+$-strategic closure (and even strong ${<}\,\lambda^+$-distributivity) is sufficient to build suitable generic filters, the same is not true for ${<}\,\lambda^+$-strategic closure.
	\end{mybem}
	
	\printbibliography
	
\end{document}